\documentclass[12pt]{article}

\usepackage{amsmath}
\usepackage{relsize}
\usepackage{amsthm}
\usepackage{amsfonts}
\usepackage{float}
\usepackage[export]{adjustbox}
\usepackage{amssymb}
\usepackage{listings}
\usepackage{color}
\usepackage{graphicx}
\usepackage{cite}
\graphicspath{{images/}}
\newcommand{\vinclude}[1]{\begingroup
\setbox0=\hbox{\includegraphics{#1}}
\parbox{\wd0}{\box0}\endgroup}

\lstdefinelanguage{Sage}[]{Python}{morekeywords={False,sage,True},sensitive=true}
\definecolor{dblackcolor}{rgb}{0.0,0.0,0.0}
\definecolor{dbluecolor}{rgb}{0.01,0.02,0.7}
\definecolor{dgreencolor}{rgb}{0.2,0.4,0.0}
\definecolor{dgraycolor}{rgb}{0.3,0.3,0.3}

\renewcommand{\emph}[1]{{\dblack{#1}}}

\usepackage{tikz}
\usetikzlibrary{arrows}
\usetikzlibrary{arrows.meta}
\usetikzlibrary{svg.path}
\usetikzlibrary{calc}
\usetikzlibrary{decorations.markings}
\usetikzlibrary{arrows.meta}

\usepackage{pgf}
\usepackage{mathrsfs}
\usepackage{bbm}
\usepackage{array}
\usepackage{fdsymbol}
\usepackage{graphicx}
\usepackage{paralist}

\newtheorem{theorem}{Theorem}
\newtheorem{mydef}{Definition}
\newtheorem{lemma}{Lemma}
\newtheorem{prop}{Proposition}
\newtheorem{remark}{Remark}

\newcommand{\lcap}{\boldmath$\langle$\unboldmath}
\newcommand{\rcap}{\boldmath$\rangle$\unboldmath}
\newcommand{\dotmap}{$\bullet$}

\title{Planar algebra presentations of $\text{URep}_{\mathbb{C}}(\mathbb{C}^+)$ and $\text{URep}_{\mathbb{F}_p}(\mathbb{F}_p^+)$}
\author{Ryan Vitale}

\begin{document}
\maketitle









\begin{abstract}
We give presentations of the planar algebra of unipotent representations of the groups $\mathbb{C}$ and $\mathbb{F}_p$ under addition using jellyfish and light leaf style arguments. These are some of the most natural examples of non-semisimple planar algebras. For the characteristic $p$ family of examples, a new generator appears in arbitrarily large box spaces as $p$ increases. We point toward future directions in getting results on first and second fundamental theorems for rings of vector invariants, as well as generalization of the examples given.
\end{abstract}

\def \low {\textbf{0}}
\def \high {\textbf{1}}

{\let\thefootnote\relax\footnote{This material is based upon work supported by the National Science Foundation under Grant No. DMS-1454767.}}

\section{Introduction}


Let $\mathbb{C}^+$ be the Lie group $\mathbb{C}$ under addition, and $\mathbb{F}_p^+$ the group $\mathbb{F}_p$ under addition. A unipotent representation $(V,\varphi)$ of a group $G$ is one in which $\varphi(x)-1$ is nilpotent for each $x \in G$. We consider unipotent (and smooth, in the case of $\mathbb{C}^+$) representations of $\mathbb{C}^+$ and $\mathbb{F}^+$ over the fields $\mathbb{C}$ and $\mathbb{F}_p$ in the context of planar algebras, denoting these categories by $\mathcal{C}_0 = URep_{\mathbb{C}}(\mathbb{C}^+)$ and $\mathcal{C}_p=URep_{\mathbb{F}_p}(\mathbb{F}_p^+)$. 

Planar algebras were formally defined by Vaughan Jones in \cite{PA1} and have been widely used in the study of subfactors. We aim to understand each category by looking at a subcategory which has the structure of a planar algebra, and still contains the full information of the category. These categories provide examples of non-semisimple planar algebras; currently there is not much in the literature on the non-semisimple case, but for an example see \cite{nonsemisimple}. We also are able to get information on vector invariants for $\mathbb{C}^+$ and $\mathbb{F}_p^+$. The characteristic $p$ case of vector invariants is currently an active area of research, with partial results in \cite{INVp,char2,secondchar2,richmaninv}. Our results for $\mathcal{C}_p$ give another persepective on some of these results.

Planar algebras provide a formalism for a $2$-dimensional symbolic language to express the morphisms in our categories. They give us the ability to draw these morphisms in a planar graphical language, and then reason about them as combinatorial and topological objects. A planar algebra is some collection of vector spaces called box spaces which have a multilinear associative action by the operad of planar tangles (this structure is discussed further in Sections \ref{sec:opas},  \ref{sec:ospas}, \ref{sec:ospaverts}, and for a full account see \cite{PA1}). Our main examples will come from particular cases of $Rep_k(G)$, the category of representations of a group $G$ over a field $k$. Fix some representation $V \in Rep_k(G)$ and consider representations which are $\otimes$-generated by $V$, i.e. built from $V$ using $-^\ast$ and $\otimes$, as well as the morphisms between such representations. The morphisms between objects whose factors are all $V$ and $V^\ast$ can be combined through composition, tensor product, and the evaluation and coevaluation of the duality between $V$ and $V^\ast$. These morphism spaces together with these combining operations fit together into a structure called a planar algebra, which we will describe in Section \ref{sec:opas}, and we will discuss the case of $Rep_k(G)$ in Section \ref{sec:repalg}).

In both cases $\mathcal{C}_0$ and $\mathcal{C}_p$ there is a self-dual two-dimensional representation $(V,\varphi)$ defined on the standard basis $v_0=(1,0),v_1=(0,1)$ by $\varphi_x(v_0)=v_0$ and $\varphi_x(v_1)=xv_0+v_1$. We denote by $\mathcal{C}_0(V)$ the full tensor subcategory of $\mathcal{C}_0$ with objects $V^{\otimes n}$, similarly for $\mathcal{C}_p(V)$. These subcategories contain enough information to describe each category; every indecomposable representation is a summand of some $V^{\otimes n}$ which can be recovered through a projection map.

\begin{prop} Let $J_i$ be the Jordan block of dimension $i$ with eigenvalue $1$. The indecomposable representations of $\mathcal{C}_0$ are enumerated by the sequence $(V_i,\varphi_i)_{i \in \mathbb{N}}$, where $V_i = \mathbb{C}^i$ and $\varphi_i:\mathbb{C}^+ \rightarrow GL(V_i)$ is defined by $\varphi_i(1)=J_i$.  The indecomposable objects of $\mathcal{C}_p$ are enumerated by the sequence $(V_i,\varphi_i)_{1\leq i \leq p}$ where $V_i = \mathbb{F}_p^i$, and $\varphi_i:\mathbb{F}_p^+ \rightarrow GL(V_i)$ is defined by $\varphi_i(1)=J_i$. Throughout this work we refer to $V_2$ as $V$ in both $\mathcal{C}_0$ and $\mathcal{C}_p$.
\end{prop}
\begin{proof}
This is well known, and presented by Srinivasan in \cite{repring}. For the case $\mathcal{C}_0$ the representation is defined by the image of $1$ by surjectivity of the exponential map (see the introduction to Chapter $3$), and for $\mathcal{C}_p$ the representation is determined by $1$ since $\mathbb{F}_p^+$ is cyclic.
\end{proof}

These indecomposables are all self-dual, as there is only one indecomposable up to isomorphism in each dimension. In $\mathcal{C}_0$, and in $\mathcal{C}_p$ excluding $V_p \in \mathcal{C}_p$, the indecomposable objects satisfy $V \otimes V_i \simeq V_{i-1} \oplus V_{i+1}$. In $\mathcal{C}_p$ we have $V \otimes V_p \simeq V_p \oplus V_p$. These $\otimes$ rules along with the full list of indecomposable objects show every indecomposable is in some tensor power of $V$. In particular, $V_i$ is a summand of multiplicity $1$ in $V^{\otimes i - 1}$. Since we retain all morphisms when taking the full subcategory $\otimes$ generated by $V$, there are projections $p_i:V^{\otimes i-1} \rightarrow V^{\otimes i-1}$ whose image and support are isomorphic to $V_i$ which restrict to the identity map on $V_i$, and in this sense we can see each indecomposable in our subcategory.

Presentations for $\mathcal{C}_0(V)$ and $\mathcal{C}_p(V)$ are stated in Section \ref{sec:presentations}, and are discussed in depth with proof in Chapters \ref{ch:char0} and \ref{ch:charp}. In the background we give several sections of introduction to the relevant algebraic structures and framework used in these presentations.

\subsection{Presentation theorems for $\mathcal{C}_0(V)$ and $\mathcal{C}_p(V)$}
\label{sec:presentations}
We state the two main theorems here; in the rest of the paper we give context to these statements, and proof of these statements. 

\begin{theorem}\label{thm:char0}
Let $Diag_0$ be the planar algebra over $\mathbb{C}$ with generators and relations $G_i$ and $E_i$ below. There is an isomorphism of planar algebras $Diag_0 \rightarrow Alg_0 = \mathcal{P}(Rep_{\mathbb{C}}(\mathbb{C}^+,V))$, where $V$ is the $2$-dimensional representation defined by $x \rightarrow \left( \begin{array}{c c} 1 & x \\ 0 & 1 \end{array} \right)$.
\begin{center}
\begin{tikzpicture}[scale=0.85]
\begin{scope}[decoration={
    markings,
    mark=at position 0.5 with {\arrow[scale=1.25]{latex}}}
    ]
\clip (-1,-1.25) rectangle (13,3.5);

\draw [postaction={decorate}] (0,0) circle (17pt);
\draw (1.2,0) node[scale=1.4] {$=$};
\draw (1.2,.05) node[anchor=south,scale=0.8] {$E_3$};
\draw (1.5,.05) node[anchor=west,scale=1.4] {$2$};
\draw [dotted] (-0.9,-1) rectangle (2.25,1);

\draw [-,postaction=decorate] (9.25,-0.75) to (9.25,0.75);
\draw [-,postaction=decorate] (9.75,-0.75) to (9.75,0.75);
\draw (10.25,0) node[scale=1.2] {$+$};

\draw [-,postaction={decorate}] (0,1.5) to (0,2.75);
\draw [fill=black] (0,2.75) circle (1.5pt);

\draw [-,postaction={decorate}] (8,3.25) to (8,1.5);
\draw (7,2.25) node[scale=1.4] {$=$};
\draw (7,2.3) node[anchor=south,scale=0.8] {$E_1$};

\draw [dotted] (5.5,1.25) rectangle (8.5,3.5);
\end{scope}

\draw (-0.1,2.25) node[anchor=east] {$G_1:$};

\begin{scope}[decoration={
	markings,
	mark=at position 0.5 with {\arrow[scale=1.25,thick]{[}},
	mark=at position 0.25 with {\arrow[scale=1.25]{latex}},
	mark=at position 0.8 with {\arrow[scale=1.25,>=latex]{<}}}
	]

\draw [-,postaction={decorate}] (2,3.25) to (2,1.5);
\draw [-,postaction={decorate}] (3.5,-0.75) to (3.5,0.75);
\draw [-,postaction={decorate}] (5.5,0.75) to (5.5,-0.75);

\draw (4.5,0) node[scale=1.4] {$=$};
\draw (4.5,0.05) node[scale=0.8,anchor=south] {$E_4$};
\draw (5,0) node[scale=1.2] {$-$};
\draw [dotted] (3,-1) rectangle (6,1);

\draw [dotted] (6.5,-1) rectangle (12.25,1);
\end{scope}

\draw (1.9,2.25) node[anchor=east] {$G_2:$};

\begin{scope}[decoration={
	markings,
	mark=at position 0.5 with {\arrow[scale=1.25,thick]{]}},
	mark=at position 0.8 with {\arrow[scale=1.25]{latex}},
	mark=at position 0.25 with {\arrow[scale=1.25,>=latex]{<}}}
	]
\draw [-,postaction={decorate}] (4,3.25) to (4,1.5);

\draw [-,postaction={decorate}] (10,3.25) to (10,1.5);
\draw [fill=black] (10,3.25) circle (1.5pt);
\draw [fill=black] (10,1.5) circle (1.5pt);
\draw (11,2.25) node[scale=1.4] {$=$};
\draw (11,2.3) node[anchor=south,scale=0.8] {$E_2$};
\draw (11.5,2.3) node[scale=1.3,anchor=west] {$0$};
\draw [dotted] (9.5,1.25) rectangle (12.25,3.5);

\end{scope}

\begin{scope}[decoration={markings,
	mark=at position 0.35 with {\arrow[scale=1.25,thick]{]}},
	mark=at position 0.55 with {\arrow[scale=1.25,>=latex]{>}},
	mark=at position 0.2 with {\arrow[scale=1.25,>=latex]{<}}}
	]
\draw [-,postaction=decorate] (12,.75) to [controls= +(-90:0.75) and +(-90:0.75)] (10.75,0.7);
\end{scope}
\begin{scope}[decoration={markings,
	mark=at position 0.35 with {\arrow[scale=1.25,thick]{]}},
	mark=at position 0.55 with {\arrow[scale=1.25,>=latex]{<}},
	mark=at position 0.2 with {\arrow[scale=1.25,>=latex]{>}}}
	]
\draw [-,postaction=decorate] (10.75,-0.75) to [controls= +(90:0.75) and +(90:0.75)] (12,-0.75);
\end{scope}

\begin{scope}[decoration={
	markings,
	mark=at position 0.3 with {\arrow[scale=1.25]{latex}},
	mark=at position 0.8 with {\arrow[scale=1.25]{latex}}}
	]

\draw [-,postaction={decorate}] (6.75,-0.75) to (8,0.75);
\draw [-,postaction={decorate}] (8,-0.75) to (6.75,0.75);	

\end{scope}


\draw [-,postaction={decoration={markings,mark=at position 0.35 with {\arrow[scale=1.25,thick]{[}},mark=at position 0.65 with {\arrow[scale=1.25,thick]{]}},mark=at position 0.2 with {\arrow[scale=1.25,>=latex]{<}},mark=at position 0.8 with {\arrow[scale=1.25,>=latex]{<}},mark=at position 0.53 with {\arrow[scale=1.25,>=latex]{>}}},decorate}] (6,1.5) to (6,3.25);

\draw (3.9,2.25) node[anchor=east] {$G_3:$};

\draw (8.5,0) node[scale=1.4] {$=$};
\draw (8.5,0.05) node[scale=0.8,anchor=south] {$E_5$};
\end{tikzpicture}
\end{center}
\end{theorem}

\begin{theorem}Working over $\mathbb{F}_p$ if we add the generator $G_4$, and the relations $E_6, E_{s_1}, E_{s_2}$ to $Diag_0$, we get a planar algebra $Diag_p$. In $E_6$, $Sym_{p-1}$ is the symmetrizer on $p-1$ strands. There is an isomorphism of planar algebras $Diag_p \rightarrow Alg_p = \mathcal{P}(Rep_{\mathbb{F}_p}(\mathbb{F}^+,V))$, where $V$ is the $2$-dimensional representation defined by $x \rightarrow \left( \begin{array}{c c} 1 & x \\ 0 & 1 \end{array} \right)$.

$$G4: \underbrace{\vinclude{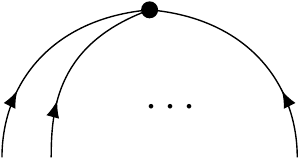}}_{2p-1}$$

\begin{center}
\vinclude{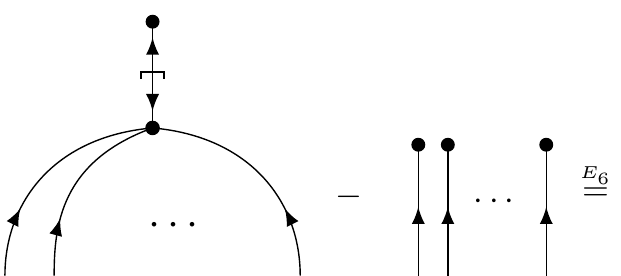} \vinclude{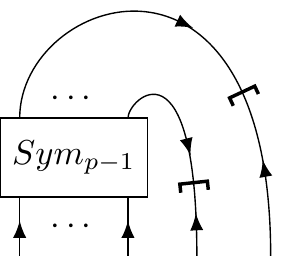}
\end{center}

\begin{center}
\vinclude{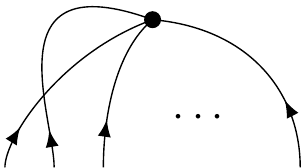} $\overset{\scriptscriptstyle{E_{s_1}}}{=}$ \vinclude{jellyfish} $\overset{\scriptscriptstyle{E_{s_2}}}{=}$ \vinclude{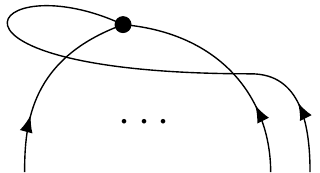}
\end{center}

\end{theorem}

\def\s{\hspace{0.5mm}}
The relations $E_{s_1}$ and $E_{s_2}$ are explicit relations that make the jellyfish symmetric, as the transposition $(1 \s 2)$ and cycle $(1 \s 2 \cdots n)$ generate the symmetric group $S_n$. Application of a crossing to any pair of legs of the jellyfish acts trivially.

\section{Background}
In this section we first state a few elementary lemmas and definitions from linear algebra and automata theory. This will provide useful structure and terminology for the proofs of the main theorems. We then discuss the algebraic framework of oriented symmetric planar algebras used for the presentations in the main theorems of Section \ref{sec:presentations}. In particular we show how a subcategory of a category of representations that is tensor-generated by a single object can be given the structure of a planar algebra. Finally, we discuss the main techniques used in proving the main theorems by an outline of the proof process. As a technical note we discuss the use of disorientation markings (oriented strands and orientation reversing brackets) in our diagrams. This explains why we use oriented strands and a vertex for the isomorphism $\varphi:V \rightarrow V^\ast$ to define Disoriented Temperley-Lieb instead of the more typically used unoriented planar algeba Temperley-Lieb first described in \cite{TL}.
\subsection{Linear Algebra}
In this section we fix some notation and prove two useful lemmas. We will deal frequently with tensor powers of two dimensional vector spaces, and the following definition will be relevant.

\begin{mydef}\label{bits}
Suppose $V$ is a two-dimensional vector space over $k$ with basis $(v_0,v_1)$. Then $V^{\otimes n}$ has basis $Z_n=\{v_0,v_1\}^{\otimes n}$ indexed by $B=\{0,1\}^{\times n}$. Define the \textbf{length} of a basis vector by its image under $l_n:Z_n \rightarrow \mathbb{N}$, where $l_n(z_b) = \sum_i b_i$. We denote the basis vectors of length $j$ in $V^{\otimes n}$ by $Z_n(j)$, and note that $|Z_n(j)|={{n} \choose {j}}$.
\end{mydef}

\def \low {\textbf{0}}
\def \high {\textbf{1}}

When it simplifies notation, we use concatenation or $\cdot$ for $\otimes$, denote $v_0$ by \low, denote $v_1$ by \high, and denote their dual vectors by $\overline{\high}$ and $\overline{\low}$. As an example of this notation convention, and of Definition \ref{bits} when $n=3$, we have $Z_3(0)=\{ \low\low\low \}, Z_3(1)=\{ \low\low\high , \low\high\low , \high\low\low \},Z_3(2)=\{ \low\high\high , \high\low\high , \high\high\low \},Z_3(3)=\{ \high\high\high \}$.

We will care about order, so will work with sequences of vectors instead of sets. Let $T:V \rightarrow W$ be a map of vector spaces, and $X=(x_i)$ a sequence in $V^n$. By $\text{span}(X)$ we mean the span of the vectors $x_1,...,x_n$ in $V$. By $T(X)$ we mean the sequence $(T(x_i))$, and when we say $T(X)$ is independent, we mean the vectors $T(x_1),...,T(x_n)$ are linearly independent in $W$. To clarify some of the arguments and language used later, we prove two elementary statements from linear algebra.

\begin{lemma} Let $T:V \rightarrow W$ be a map of vector spaces with $dim(W)=n$. Suppose we have sequences $X,Y \in V^n$ such that $span(Y)=V$, and $T(X)$ is linearly independent. Then $T$ is an isomorphism, and both $X$ and $Y$ are bases of $V$. 
\end{lemma}
\begin{proof}
 Since $\text{span}(Y)=V$, $\dim(V) \leq n$. Since $T(X)$ is independent, $X$ is independent, so $\dim(V)\geq n$. Then $\dim(V)=n=\dim(W)$, and $T(X)$ has length $n$, so $T$ is an isomorphism. The sequences $X$ and $Y$ are bases of $V$ since they are sequences of size $n$ that are independent in $V$ and span $V$, respectively.
\end{proof}

\begin{lemma}\label{indep} Let $(S,<)$ be a finite totally ordered set, and $V$ a vector space. If there are maps $f:S \rightarrow V$ and $g:S \rightarrow V^*$ such that for all $x,y \in S$:
\begin{enumerate}
\item $g(x)(f(x)) \neq 0$
\item $x<y \implies g(x)(f(y))=0$
\end{enumerate}
Then both the values of $f$ and the values of $g$ are linearly independent.
\end{lemma}
\begin{proof}
For simplicity and without loss of generalization replace $S$ with the ordered set $(1,2,\dots,n)$. Consider the $n \times n$ matrix $A$ defined by $A_{i,j}=g(i)(f(j))$. Condition $(1)$ of the lemma implies all entries on the main diagonal of $A$ are non-zero. Condition $(2)$ implies $A$ is upper triangular, so together $(1)$ and $(2)$ imply $A$ is invertible. Any linear dependence among the rows of $A$ implies a linear dependence among the values of $g$, and a dependence among columns of $A$ implies a dependence among the values of $f$, so we have our result.
\end{proof}

\subsection{Formal languages and automata}

An \textbf{alphabet} can be any set $\Sigma$. A \textbf{word} of length $n$ over $\Sigma$ is some element of $\Sigma^n$, and a \textbf{language} over $\Sigma$ is some subset $L \subset \Sigma^*$, where $\Sigma^*=\bigcup_{i \in \mathbb{N}} \Sigma^m$. A \textbf{segment} of a word $w$ is some contiguous subsequence of $w$, and is called initial if it starts at the beginning of $w$. Denote the sublanguage of words of length $n$ in $L$ by $L(n)$. We will consider languages that are defined by automata.

\begin{mydef}
An \textbf{automaton} $M=(Q,A,\tau)$ over the alphabet $\Sigma$ is a rooted graph with directed edges labelled by $\Sigma$, where:
\begin{itemize}
\item $Q$ is the vertex set called the set of states, and the root $Q_* \in Q$ is called the start state.
\item $A$ is a subset of $Q$ called the accepting states.
\item $\tau$ is the set of directed edges labelled by $\Sigma$ called the transition function.
\end{itemize}

A word is \textbf{accepted} by $M$ if the path starting at $Q_*$ that it defines ends at an accepting state. The language $L_M$ is defined to be all words accepted by $M$.
\end{mydef}

For further reading on automata we reference \cite{AutBook}, but in the scope of this work we will just use the examples of the following subsections.

\subsection{Two infinite automata}

Define $M_0$ by $\Sigma=\{R,L\}, Q=A=\{V_i : i \in \mathbb{Z}_{>0}\}$, with $V_1$ as the start state and $\tau$ as illustrated below. We say the \textbf{depth} of a word $w \in \Sigma^*$, denoted $d(w)$, is the number of $R$s minus the number of $L$s in $w$. The depth at $i$ of $w$, denoted $d_i(w)$, is the depth of the initial segment of $w$ of length $i$. We see $w \in L_{M_0}$ when $d_i(w) \geq 0$ for every index $i$.

\begin{center}
\includegraphics{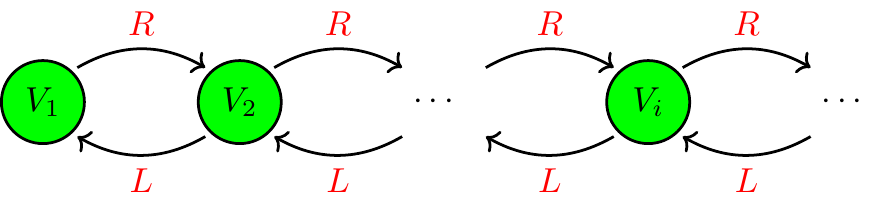}
\end{center}

Define $N_0$ by $\Sigma=\{\text{\dotmap,\lcap,\rcap\}}, Q=\{V_i : i \in \mathbb{Z}_{>0}\}, A= \{V_1\}$, with $V_1$ as the start state and $\tau$ as illustrated below. We say the \textbf{depth} of a word $w \in \Sigma^*$, denoted $d(w)$, is the number of \lcap s minus the number of \rcap s in $w$. The depth at $i$ of $w$ is the depth of the initial segment of $w$ of length $i$. We see that $w \in L_{N_0}$ when 

\begin{enumerate}
\item $d_i(w) \geq 0$ for every index $i$, and $d(w)=0$.
\item A \dotmap\hspace{1.5mm}can only appear at depth $0$ (i.e. at an index $i$ in $w$ where $d_i(w)=0$) .
\end{enumerate}

\begin{center}
\includegraphics{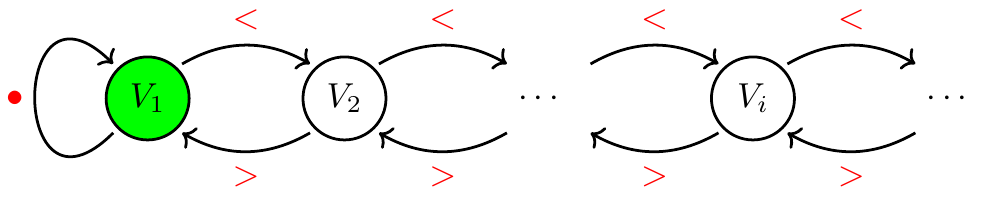}
\end{center}

\begin{prop}\label{aut0bij} There is a length preserving bijection between $L_{M_0}$ and $L_{N_0}$.
\end{prop}
\begin{proof} We will say an $R$ at index $i$ in $w \in L_{M_0}$ is depth-increasing if $d_i(w) \leq d_j(w)$ for all $j > i$. To define our bijection in the direction $L_{M_0} \xrightarrow{f} L_{N_0}$, replace each $R$ with \dotmap\hspace{1.5mm}if $R$ is depth-increasing, and with \lcap\hspace{1.5mm}if $R$ is not depth-increasing. Replace each $L$ with \rcap. Since all depth-increasing $R$s are replaced with \dotmap\hspace{1.5mm}the images of $f$ will have depth $0$, and every \dotmap\hspace{1.5mm}will occur at depth $0$. Further since $w$ has nonegative depth at each index, so will $f(w)$, so $f(w)$ will be accepted by $N_0$. The map in the reverse direction sends \dotmap\hspace{1.5mm} and \lcap\hspace{1.5mm}to $R$, and \rcap\hspace{1.5mm}to $L$.
\end{proof}

\subsection{Two finite automata}

Define $M_p$ by $\Sigma=\{R,L,A,B\}, Q=A=\{V_i : i \in \mathbb{Z}_{>0}\}$, with $V_1$ as the start state and $\tau$ as illustrated below. We write a word $w \in L_{N_p}$ as the concatenation of two parts $w=\overset{\leftarrow}{w} \cdot \overset{\rightarrow}{w}$, with $\overset{\leftarrow}{w} \in \{R,L\}^*$ and $\overset{\rightarrow}{w} \in \{A,B\}^*$. The notion of depth from $M_0$ still makes sense on $w$.

\begin{center}
\includegraphics{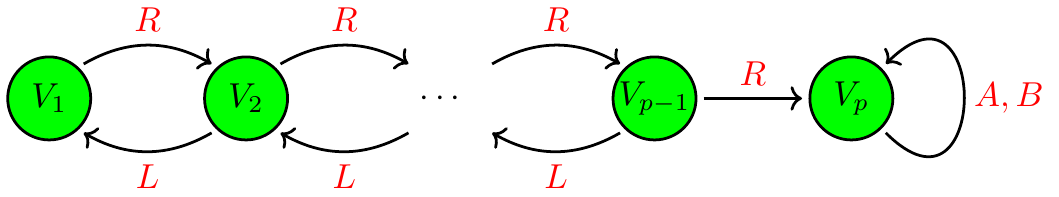}
\end{center}

We see that $w \in L_{M_p}$when:

\begin{enumerate}
\item $0 \leq d_i(w) \leq p-1$ for every index $i$.
\item $\overset{\rightarrow}{w}$ is empty if $d(\overset{\leftarrow}{w}) \neq p-1$.

\end{enumerate}

Define $N_p$ by $\Sigma=\{*,\text{\dotmap,\lcap,\rcap\}}, Q=\{V_i : i \in \mathbb{Z}_{>0}\}, A = \{V_0\}$ with $V_0$ as the start state and $\tau$ as illustrated below. The notion of depth from $N_0$ still makes sense.

\begin{center}
\includegraphics{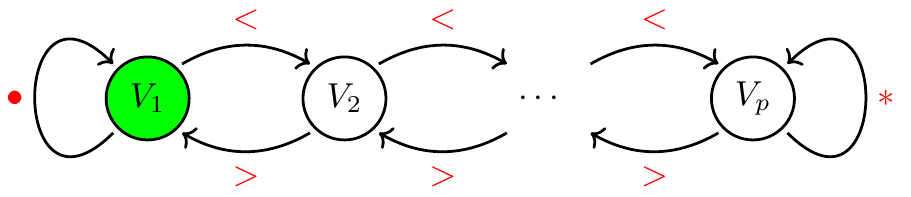}
\end{center}

We see that $w \in L_{N_p}$ when:
\begin{enumerate}
\item $0 \leq d_i(w) \leq p-1$ for every index $i$, and $d(w)=0$.
\item A \dotmap\hspace{1.5mm}can only appear at depth $0$.
\item A $*$ can only appear at depth $p-1$.
\end{enumerate}

As a visual example of these languages, we define mountains and plateaus. We then show $L_{M_p}$ corresponds bijectively to mountains, $L_{N_p}$ corresponds bijectively to plateaus, and that there is a length preserving bijection between $L_{M_p}$ and $L_{N_p}$.

\begin{mydef} A \textbf{mountain} of height $p-1$ is a walk on $\mathbb{N} \times \mathbb{Z}$ starting at $(0,0)$ in steps of $s_1=(1,1)$ and $s_{-1}=(1,-1)$, such that the walk stays above height $0$ until the first time it reaches height $p-1$.  We denote the set of mountains of height $p-1$ by $Mount_{p-1}$.
\end{mydef}

\begin{prop} There is a length preserving bijection from $Mount_{p-1}$ to $L_{M_p}$.
\end{prop}
\begin{proof} Immediate from the definition, taking $R$ to be $s_1$ and $L$ to be $s_{-1}$ during $\overset{\leftarrow}{w}$, then $B$ to be $s_1$ and $A$ to be $s_{-1}$ during $\overset{\rightarrow}{w}$.
\end{proof}

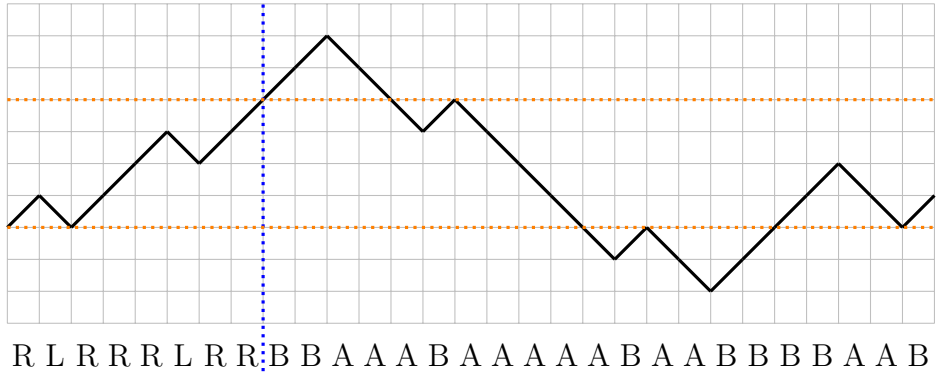
\begin{figure}[H]
\centering
\begin{tikzpicture}[scale=0.85]
\def\R#1#2{\draw[very thick] (0.5*#1,0.5*#2) to (0.5*#1+0.5,0.5*#2+0.5); \draw (0.5*#1 +0.25,-2) node {R};}
\def\L#1#2{\draw[very thick] (0.5*#1,0.5*#2) to (0.5*#1+0.5,0.5*#2-0.5); \draw (0.5*#1 +0.25,-2) node {L};}
\def\B#1#2{\draw[very thick] (0.5*#1,0.5*#2) to (0.5*#1+0.5,0.5*#2+0.5); \draw (0.5*#1 +0.25,-2) node {B};}
\def\A#1#2{\draw[very thick] (0.5*#1,0.5*#2) to (0.5*#1+0.5,0.5*#2-0.5); \draw (0.5*#1 +0.25,-2) node {A};}
\draw[step=0.5,color=lightgray] (0,-1.5) grid (14.5,3.5);

\R{0}{0}\L{1}{1}\R{2}{0}\R{3}{1}\R{4}{2}\L{5}{3}\R{6}{2}\R{7}{3}\B{8}{4}\B{9}{5}\A{10}{6}\A{11}{5}\A{12}{4}\B{13}{3}\A{14}{4}\A{15}{3}\A{16}{2}\A{17}{1}\A{18}{0}\B{19}{-1}\A{20}{0}\A{21}{-1}\B{22}{-2}\B{23}{-1}\B{24}{0}\B{25}{1}\A{26}{2}\A{27}{1}\B{28}{0}

\draw[color=blue, very thick, dotted] (4,-2.25) to (4,3.5);
\draw[color=orange,very thick, dotted] (0,0) to (14.5,0);
\draw[color=orange,very thick, dotted] (0,2) to (14.5,2);
\end{tikzpicture}
\caption{A mountain of length $29$ and height $4$ drawn above its corresponding word. The dotted blue line marks the end of the $R,L$ part of the word, and the heights $0$ and $4$ are marked in orange.}
\label{mountain}
\end{figure}

\begin{mydef} A \textbf{plateau} of height $p-1$ is a walk on $\mathbb{N} \times \{0,1,\dots,p-1\}$ from $(0,0)$ to $(n,0)$ in steps of $s_0=(1,0),s_1=(1,1),$ and $s_{-1}=(1,-1)$, where we require that $s_0$ occurs only at height $0$ or $p-1$. We denote the set of plateaus of height $p-1$ by $Plat_{p-1}$.
\end{mydef}

\begin{prop} There is a length preserving bijection from $Plat_{p-1}$ to $L_{N_p}$.
\end{prop}
\begin{proof} Send \dotmap\hspace{1.5mm} and $*$ to $s_0$, send \lcap\hspace{1.5mm}to $s_1$, and send \rcap\hspace{1.5mm}to $s_{-1}$. In the other direction we just need to distinguish $*$ from \dotmap, choosing $*$ if the height of the walk is currently $p-1$, and \dotmap\hspace{1.5mm} if it is $0$.
\end{proof}

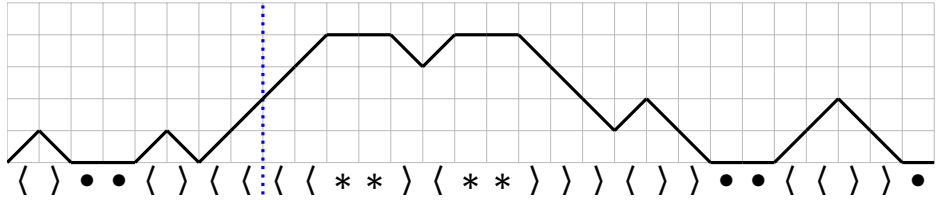
\begin{figure}[H]

\centering
\begin{tikzpicture}[scale=0.85]
\def\up#1#2{\draw[very thick] (0.5*#1,0.5*#2) to (0.5*#1+0.5,0.5*#2+0.5); \draw (0.5*#1 +0.25,-0.3) node {\text{\lcap}};}
\def\down#1#2{\draw[very thick] (0.5*#1,0.5*#2) to (0.5*#1+0.5,0.5*#2-0.5); \draw (0.5*#1 +0.25,-0.3) node  {\text{\rcap}};}
\def\flatL#1#2{\draw[very thick] (0.5*#1,0.5*#2) to (0.5*#1+0.5,0.5*#2); \draw (0.5*#1 +0.25,-0.3) node  {\text{\dotmap}};}
\def\flatH#1#2{\draw[very thick] (0.5*#1,0.5*#2) to (0.5*#1+0.5,0.5*#2); \draw (0.5*#1 +0.25,-0.3) node {$*$};}

\clip (0,-0.5) rectangle (14.5,2.75);
\draw[step=0.5,color=lightgray] (0,0) grid (14.5,2.5);
\draw[dotted,color=blue,very thick] (4,-0.5) to (4,2.5);

\up{0}{0}\down{1}{1}\flatL{2}{0}\flatL{3}{0}\up{4}{0}\down{5}{1}\up{6}{0}\up{7}{1}\up{8}{2}\up{9}{3}\flatH{10}{4}\flatH{11}{4}\down{12}{4}\up{13}{3}\flatH{14}{4}\flatH{15}{4}\down{16}{4}\down{17}{3}\down{18}{2}\up{19}{1}\down{20}{2}\down{21}{1}\flatL{22}{0}\flatL{23}{0}\up{24}{0}\up{25}{1}\down{26}{2}\down{27}{1}\flatL{28}{0}
\end{tikzpicture}
\caption{A plateau of length $29$ and height $5$ drawn above its corresponding word.}
\label{plateau}
\end{figure}

The mountain and plateau of Figure \ref{mountain} and Figure \ref{plateau} above are in correspondence via the bijection in the proof of Proposition \ref{mountplat} below.

\begin{prop}\label{mountplat} There is a length preserving bijection between $L_{M_p}$ and $L_{N_p}$.
\end{prop}
\begin{proof} We show the bijection between mountains and plateaus. If our mountain/plateau never attains height $p-1$, $\overset{\rightarrow}{w}$ is empty, and we use the bijection in the proof of Proposition \ref{aut0bij}. Otherwise, we use the following process. On a plateau moving left to right, call the left endpoint of any step of type $s_0$ a \textbf{hinge}. Now swing up (replace $s_0$ with $s_1$, moving the rest of the plateau up $1$ unit) any hinge at height $0$, and swing down (replace $s_0$ with $s_{-1}$, moving the rest of the plateau down $1$ unit) any hinge at height $p-1$. This gives a mountain. In the opposite direction we need to be able to identify steps along the mountain which came from swinging a hinge. To do this, we traverse the mountain from right to left keeping a counter called \textbf{*gap} which starts at $0$. Any time we would move below our starting height, we swing that step up along its right endpoint and decrement *gap by $1$. Any time we would move to height $p$ we swing that step down along its right endpoint and increment *gap by $1$. After completing the $\overset{\rightarrow}{w}$ part of the mountain, we are at height $p-1$ with respect to the left end of the mountain. The value of *gap indicates how much above or below the left starting point the bottom of our plateau is, i.e. the plateau bottom minus *gap is the starting height of the mountain. So, we consider our height to be $p-1-$*gap, and use that to determine when to swing a hinge up when traversing $\overset{\leftarrow}{w}$. When we would move below the plateau bottom, swing that step up 1. By counting *gap we guarantee this is inverse to the process of going from a mountain to a plateau, and we have our bijection.

\begin{center}
\begin{tikzpicture}[scale=0.85]
\clip (0,-2.75) rectangle (14.5,3.75);
\draw[step=0.5,color=lightgray] (0,-2.5) grid (14.5,3.5);
\draw[dotted,color=blue,very thick] (4,-2.5) to (4,3.5);

\def\R#1#2{\draw[very thick] (0.5*#1,0.5*#2) to (0.5*#1+0.5,0.5*#2+0.5); \draw (0.5*#1 +0.25,3.25) node {R};}
\def\L#1#2{\draw[very thick] (0.5*#1,0.5*#2) to (0.5*#1+0.5,0.5*#2-0.5); \draw (0.5*#1 +0.25,3.25) node {L};}
\def\B#1#2{\draw[very thick] (0.5*#1,0.5*#2) to (0.5*#1+0.5,0.5*#2+0.5); \draw (0.5*#1 +0.25,3.25) node {B};}
\def\A#1#2{\draw[very thick] (0.5*#1,0.5*#2) to (0.5*#1+0.5,0.5*#2-0.5); \draw (0.5*#1 +0.25,3.25) node {A};}

\def\up#1#2{\draw[very thick] (0.5*#1,0.5*#2) to (0.5*#1+0.5,0.5*#2+0.5); \draw (0.5*#1 +0.25,-2.25) node {\text{\lcap}};}
\def\down#1#2{\draw[very thick] (0.5*#1,0.5*#2) to (0.5*#1+0.5,0.5*#2-0.5); \draw (0.5*#1 +0.25,-2.25) node  {\text{\rcap}};}
\def\flatL#1#2{\draw[very thick] (0.5*#1,0.5*#2) to (0.5*#1+0.5,0.5*#2); \draw (0.5*#1 +0.25,-2.25) node  {\text{\dotmap}};}
\def\flatH#1#2{\draw[very thick] (0.5*#1,0.5*#2) to (0.5*#1+0.5,0.5*#2); \draw (0.5*#1 +0.25,-2.25) node {$*$};}

\begin{scope}[color=blue]

\R{0}{-1}\L{1}{0}\R{2}{-1}\R{3}{0}\R{4}{1}\L{5}{2}\R{6}{1}\R{7}{2}\B{8}{3}\B{9}{4}\A{10}{5}\A{11}{4}\A{12}{3}\B{13}{2}\A{14}{3}\A{15}{2}\A{16}{1}\A{17}{0}\A{18}{-1}\B{19}{-2}\A{20}{-1}\A{21}{-2}\B{22}{-3}\B{23}{-2}\B{24}{-1}\B{25}{0}\A{26}{1}\A{27}{0}\B{28}{-1}
\end{scope}

\up{0}{0}\down{1}{1}\flatL{2}{0}\flatL{3}{0}\up{4}{0}\down{5}{1}\up{6}{0}\up{7}{1}\up{8}{2}\up{9}{3}\flatH{10}{4}\flatH{11}{4}\down{12}{4}\up{13}{3}\flatH{14}{4}\flatH{15}{4}\down{16}{4}\down{17}{3}\down{18}{2}\up{19}{1}\down{20}{2}\down{21}{1}\flatL{22}{0}\flatL{23}{0}\up{24}{0}\up{25}{1}\down{26}{2}\down{27}{1}\flatL{28}{0}

\end{tikzpicture}
\end{center}

\end{proof}

\begin{remark}
Being a bit loose with notation we can use the same visuals for the infinite automata. We have length preserving bijections showing $Plat_{\infty} \simeq L_{N_0} \simeq L_{M_0} \simeq Mount_{\infty}$. The proofs of the propositions of this section (when adjusted appropriately to make sense) exhibit these bijections.
\end{remark}

\subsection{Oriented Planar Algebras (OPAs)}
\label{sec:opas}
\def\marker{$\star$}

We start by describing oriented planar tangles (OPTs), which will be templates for combining elements of an oriented planar algebra. Each OPT is made up of input discs contained in an output disc, with oriented nonintersecting strands connecting the boundaries of these discs. Each disc has one boundary interval marked by $\star$ for rotational alignment. We consider two OPTs the same if there is an orientation preserving planar isotopy between them. The strand endpoints and orientations determine the \textbf{type} of a tangle, which we can specify as a list of input disc types and an output disc type. To define the type of an input disc, traverse its boundary counterclockwise from $\star$, labelling each strand intersection $+$ if the strand is incoming and $-$ if the strand is outgoing. We use the opposite orientation for the type of the output disc.
\begin{figure}[H]
\centering
\includegraphics{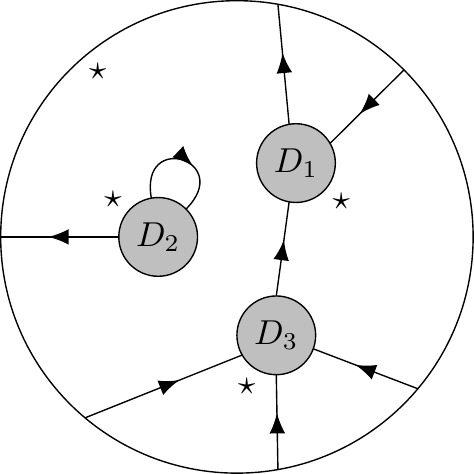}
\caption{An example of an OPT with output disc of type $(-,+,+,+,+,-)$, and input discs $D_1,D_2,$ and $D_3$ of types $(+,-,+),(-,+,-),$ and $(+,+,-,+)$ respectively.}
\label{fig:opt}
\end{figure}
The space of OPTs is a colored operad \cite{operads}, where composition is performed by inserting some tangle into a specified input disc of another tangle with each \marker \hspace{1mm}aligned. This composition is defined when the output type of one tangle is equal to the input type of the specified input disc of the other tangle.

\begin{figure}[H]
\centering
\includegraphics{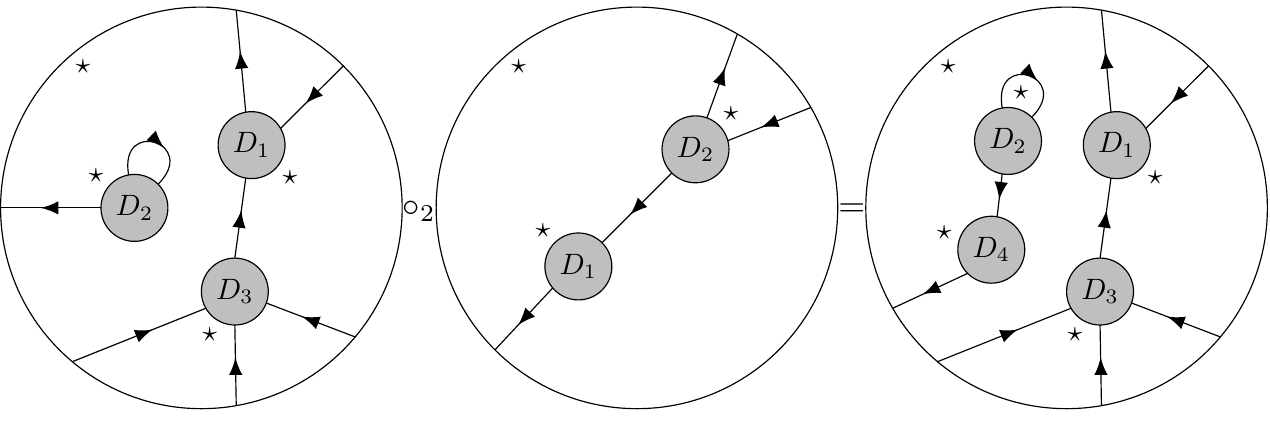}
\caption{Composition of two OPTs at the input disc $D_1$.}
\label{fig:optcomp}
\end{figure}

We give a topological definition of OPTs, and refer the reader to \cite{PA1}. We make use of the topological description while performing calculations and reasoning about oriented planar algebras.\newpage

\begin{mydef}
 An \textbf{oriented planar tangle (OPT)} consists of a smooth embedding $d:\underset{i=1..n}{\prod} D_i \rightarrow \text{int}(D_0)$ where $D_0,D_i \simeq D^2$, and a smooth embedding $s:\underset{i=1..n}{\prod} I_i \rightarrow \overline{D_0 - \coprod D_i}$ where $I_i \simeq D^1$ and we require that $s$ restricts to a map on the boundaries. The images $d |_{D_i}$ for $i=1,\dots,n$ are the \textbf{input discs}, and $D_0$ is the \textbf{output disc}. The images $s |_{I_i}$ are called \textbf{strands}, and the images of $\delta(I_i)$ are called the \textbf{strand ends}. For each disc $D_i$, the complement of the strand ends in $\delta(D_i)$ form path components, one of which is marked with (\marker) for rotational alignment. These tangles are defined up to orientation preserving planar isotopy of the strands and discs.
\end{mydef}

\begin{mydef}
An \textbf{oriented planar algebra (OPA)} is a collection of vector spaces indexed by $\underset{n \in \mathbb{N}}{\bigcup} \{+,-\}^n$ called \textbf{box spaces}, with a multilinear action by the operad of OPTs. 
\end{mydef}

This means that for each tangle $t$ of input type $(\sigma_1,\dots, \sigma_m)$ and output type $\sigma_0$ we need to define a linear map $t:\beta_{\sigma_1} \otimes \cdots \otimes \beta_{\sigma_m} \rightarrow \beta_{\sigma_0}$, e.g. for the tangle of Figure \ref{fig:opt} we would need to define a linear map $\beta_{(+,-,+)} \otimes \beta_{(-,+,-)} \otimes \beta_{(+,+,-,+)} \rightarrow \beta_{(-,+,+,+,+,-)}$.

We consider oriented Temperley-Lieb as a first example.

\subsubsection{OTL($\delta$)}
Consider the planar algebra of all formal linear combinations of oriented tangles with no input discs, and where the value of the circle with either orientation is $\delta$ (i.e. the oriented circle is equal to the empty tangle with coefficient $2$). We call this Oriented Temperley-Lieb at $\delta$ ($OTL(\delta)$), after the unoriented case first described by Temperley and Lieb in \cite{TL}, with diagrams introduced by Kauffman in \cite{kauffpoly} and the planar algebra structure implicit in \cite{Index}. We will denote the box spaces $\beta_\sigma$, indexed by elements $\sigma \in \{+,-\}^n$ where $n\in \mathbb{N}$. The index carries the information of the number of strands $n$, and the orientations of each strand at the disc boundary ordered counterclockwise from the \marker \hspace{1mm}($+$ for outgoing and $-$ for incoming). Explicitly, the box spaces are formal linear combinations of diagrams comprised of non-intersecting oriented arcs. For example, the box space $\beta_{(+,-,+,-)}$ is generated as a vector space by the following $2$ tangles.

\begin{figure}[H]
\centering
\includegraphics{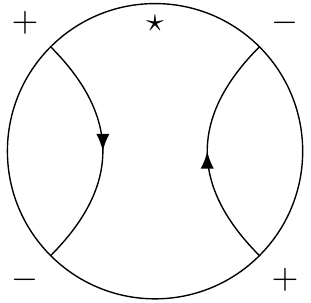} \hspace{10mm} \includegraphics{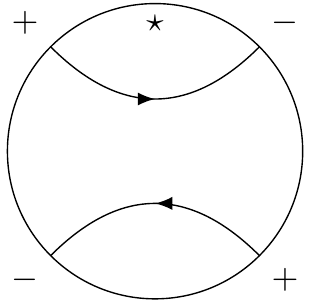}
\caption{A basis for the vector space $\beta_{(+,-,+,-)}$.}
\end{figure}

We combine elements of the OPA via composition in the operad of OPTs. Composition is performed by gluing diagrams into the input of a tangle and distributing whenever inputing a linear combination. Whenver an oriented circle appears, it is removed and the coefficient of that diagram is multiplied by $\delta$.

\begin{figure}[H]
\centering

\vinclude{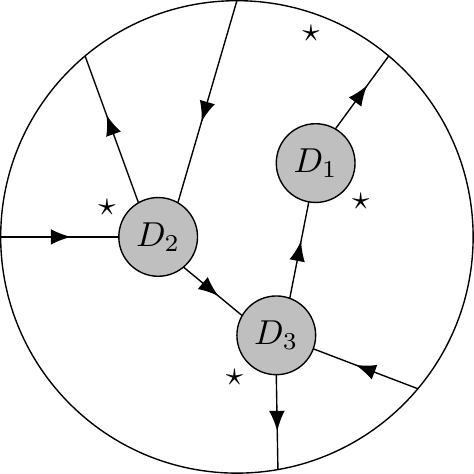} \vinclude{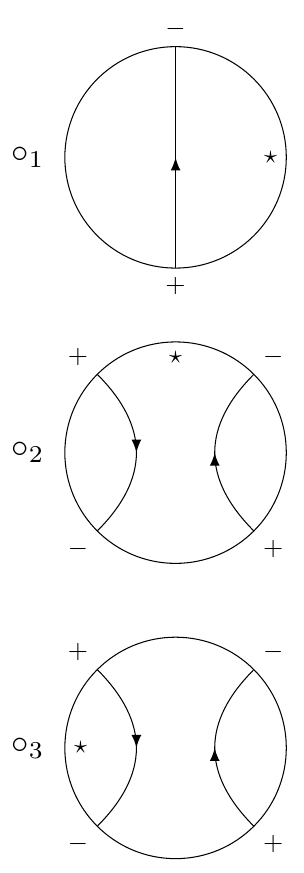} $=$ \vinclude{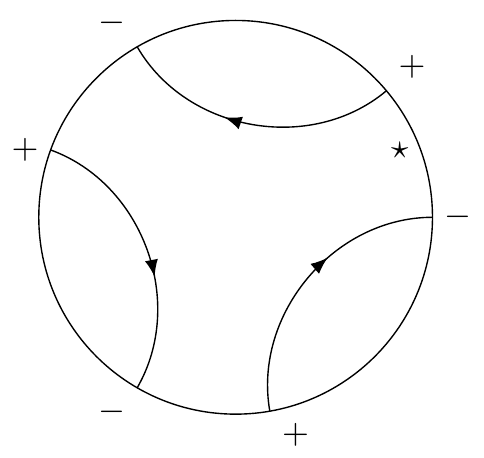}

\caption{The process and result of applying an element of the OPT to a tuple from $\beta_{(-,+)}\times \beta_{(+,-,+,-)} \times \beta_{(-,+,-,+)}$ to get an element of $\beta_{(+,-,+,-,+,-)}$.}
\end{figure}

\subsection{Oriented Symmetric Planar Algebras (OSPAs)}
\label{sec:ospas}
An oriented symmetric planar tangle (OSPT) is an output disc containing finitely many input discs, whose boundaries are connected by finitely many oriented and possibly transversely intersecting strands. 

\begin{mydef}
 An \textbf{oriented symmetric planar tangle (OSPT)} consists of an embedding $d:\underset{i=1..n}{\prod} D_i \rightarrow \text{int}(D_0)$ where $D_0,D_i \simeq D^2$, and an immersion $s:\underset{i=1..n}{\prod} I_i \rightarrow \overline{D_0 - \coprod D_i}$ where $I_i \simeq D^1$ and we require that $s$ restricts to a map on the boundaries. The images $d |_{D_i}$ for $i=1,\dots,n$ are the \textbf{input discs}, and $D_0$ is the \textbf{output disc}. The images $s |_{I_i}$ are called \textbf{strands}, and the images of $\delta(I_i)$ are called the \textbf{strand ends}. For each disc $D_i$, the complement of the strand ends in $\delta(D_i)$ form path components, one of which is marked with (\marker) for rotational alignment. We say two OSPTs are equivalent if one can be related to the other by any combination of regular isotopy (isotopy of immersions), naturality of input discs, and Reidemeister moves. We refer to this equivalence relation as \textbf{symmetric isotopy}.
\end{mydef}

\begin{figure}[H]
\centering
\vinclude{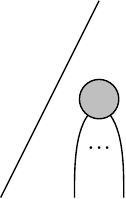} = \vinclude{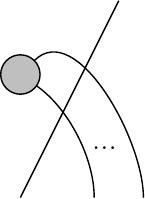} \vspace{5mm}

\hspace*{\fill}
\hfill \vinclude{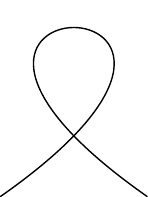} = \vinclude{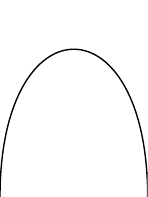} \hfill \vinclude{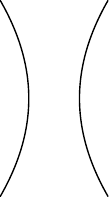} = \vinclude{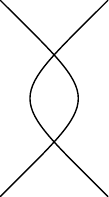} \hfill \vinclude{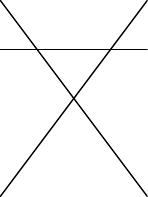} = \vinclude{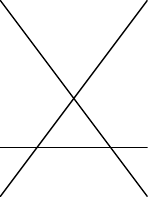}
\hspace*{\fill}
\caption{The first equality illustrates naturality of an input disc, and the other three equalities illustrate Reidemeister I, II, and III. These equalitities should hold for any choice of strand orientations.}
\label{fig:symiso}
\end{figure}

The space of OSPTs form an operad in the same way as the space of OPTs.
\begin{mydef}
An \textbf{oriented symetric planar algebra (OSPA)} is a collection of vector spaces indexed by $\underset{n \in \mathbb{N}}{\bigcup} \{+,-\}^n$ called \textbf{box spaces}, with a multilinear action by the operad of OSPTs. 
\end{mydef}

In an OSPA over the field $k$ there is a distinguished element of the $(+,+,-,-)$-box space called the \textbf{crossing}. We draw this element as a pair of crossing strands \includegraphics[scale=0.5]{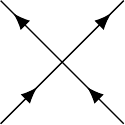} which must satisfy the Reidemeister moves and naturality. This follows from the action by the operad of OSPTs, as the tangle $t$ which is a pair of crossing strands defines a map $t:k \rightarrow \beta_{(+,+,-,-)}$. The crossing is taken to be the image $t(1)$.

\begin{lemma}
Any OSPA can be considered an OPA by restricting to the action of planar diagrams. To endow an OPA with the structure of an OSPA we can choose a crossing, and check that it satisfies naturality and the Reidemeister moves.
\end{lemma}
\begin{proof}
Follows from the definitions.
\end{proof}

\subsection{Rep$_k(G,V)$ as an OSPA}
\label{sec:repalg}
Fix a group $G$, a field $k$, and some object $V\in \text{Rep}_k(G)$, the category of $k$-linear representations of $G$. Consider $\text{Rep}_k(G,V)$, the full subcategory of $\text{Rep}_k(G)$ whose objects are finite tensor products of $V$ and $V^{\ast}$. From $\text{Rep}_k(G,V)$ we will define a planar algebra $\mathcal{P}(\text{Rep}_k(G,V))$, and from this construct a category $\mathcal{C}(\mathcal{P}(\text{Rep}_k(G,V)))$ equivalent to $\text{Rep}_k(G,V)$.

\begin{mydef} 
Let $\mathcal{P}(\text{Rep}_k(G,V))$ be the OSPA whose box spaces are the morphism spaces of $\text{Rep}_k(G,V)$. 
\end{mydef}

For this to be an OSPA we need to show there is an action by the space of OPTs on the morphism spaces of $\text{Rep}_k(G,V)$. In this direction we first explain how to interpret an OPT with no input discs as a map of representations. We will use the following definition.

\begin{mydef} Given a sequence $\sigma = (\sigma_i) \in \{+,-\}^n$, define $V^{\otimes \sigma} = \bigotimes_{\sigma} V^{\sigma_i}$, where $V^{+}=V$ and $V^{-}=V^{\ast}$. For any such $\sigma$, define the dual to be the negative reversed sequence $\sigma^{\ast}=(-\sigma_n,\dots,-\sigma_1)$ where $-+=-$ and $--=+$.
\end{mydef}

Every OPT has a sequence of oriented points $\sigma$ along its boundary. We treat positively oriented points as copies of $V$, negatively oriented points as copies of $V^\ast$, and a tangle with no input discs as a morphism in $Hom(V^\sigma , \mathbbm{1})$. A tangle with no input discs is comprised of (possibly crossing) strands connecting pairs of boundary points. We then need to be able to interpret arcs and strand crossings as morphisms in $\text{Rep}_k(G,V)$. For any group $G$ and representation $V$, we have $G$-invariant evaluation, coevaluation, and crossing maps $\epsilon: V^{\ast} \otimes V \rightarrow \mathbbm{1}$, $\eta: \mathbbm{1} \rightarrow V \otimes V^{\ast}$, $\tau_{A,B}: A \otimes B \rightarrow B \otimes A$, as well as the identity maps on $V$ and $V^\ast$.

\begin{figure}[H]
\hspace*{\fill}
\hspace{10mm} $\underset{\epsilon}{\vinclude{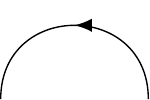}}$ \hfill  $\underset{\eta}{\vinclude{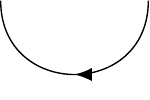}}$ \hfill  $\underset{\tau_{V,V^\ast}}{\vinclude{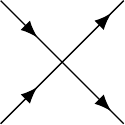}}$ \hfill $\underset{1_V}{\vinclude{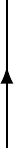}}$ \hfill  $\underset{1_{V^{\ast}}}{\vinclude{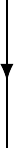}}$\hspace{10mm}
\hspace*{\fill}

\caption{Reading upwards, interpret each diagram above as the morphism it is labelled by.}
\end{figure}

\begin{remark} There are really two evaluation (coevaluation) maps depending on the orientation of the arc. In abuse of notation we call these both $\epsilon$ ($\eta$) regardless of orientation, noting that in the opposite orientation we are interpreting the diagram by pre (post) composing with the crossing map.
\end{remark}

The OSPT is supposed to be defined up to symmetric isotopy, so the interpretation of a tangle as a morphism shouldn't change under planar isotopy or Reidemeister moves. Checking that Reidemeister moves hold is straightforward. Planar isotopy follows from the equalities of morphisms $1_V = (1_V \otimes \epsilon) \circ (\eta \otimes 1_V) = (\epsilon \otimes 1_V) \circ (1_V \otimes \eta)$, and similarly for $1_{V^{\ast}}$.

\begin{figure}[H]
\hspace*{\fill}
\hspace{10mm} $\underset{(\epsilon \otimes 1_V) \circ (1_V \otimes \eta)}{\vinclude{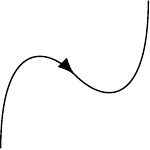}}$ \hfill $=$ \hfill $\underset{1_V}{\vinclude{idmap}}$ \hfill $=$ \hfill $\underset{(1_V \otimes \epsilon) \circ (\eta \otimes 1_V)}{\vinclude{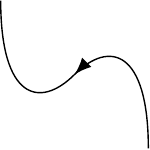}}$ \hspace{10mm}
\hspace*{\fill}

\caption{The above planar isotopic diagrams give the same map of representations}
\end{figure}

Interpret gluing of strands to be composition and placing two morphisms next to each other to be the $\otimes$ product of those morphisms, giving an action of the operad of OPTs on the box spaces.

\begin{figure}[H]
\centering
\includegraphics{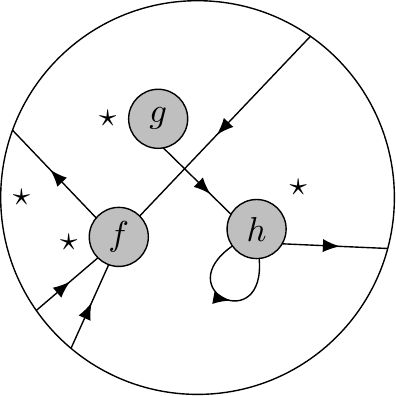}
\caption{The above diagram is the result in $Hom(V^{(+,+,-,+,-)}, \mathbbm{1})$ of the action of an OPT on $f \otimes g \otimes h \in Hom(V^{(+,+,+)},V) \otimes Hom(V^{(-)},\mathbbm{1}) \otimes Hom(V^{(-,+)},V^{(-,+)})$.}
\end{figure}

To illustrate we follow a possible set of steps in computing the value of a vector under the map determined by the diagram in the figure above. We will denote the output components of a map by indices, so that for example $h$ has two components $h_1$ and $h_2$. Take the vector $w_0 \otimes w_1 \otimes 1  \otimes \phi_2 \otimes w_3 \otimes \phi_4$ where we have included a copy of $k$ in the third factor so that we can perform the coevaluation at the inputs of $h$. For simplicity suppose $V$ is two dimensional with basis $(v_0,v_1)$. Generally we have $\eta(1)= \underset{v_i \in X}{\sum} v_i^\ast \otimes v_i$ where $X$ is a basis for $V$.

\begin{enumerate}
\item Start. $$w_0 \otimes w_1 \otimes 1 \otimes \phi_2 \otimes w_3 \otimes \phi_4$$
\item Apply coevaluation to $1$. $$w_0 \otimes w_1 \otimes (v_0^\ast \otimes v_0 + v_1^\ast \otimes v_1) \otimes \phi_2 \otimes w_3 \otimes \phi_4$$
\item Apply $h$ to the result of coevaluation from the prior step.
\begin{eqnarray*} &&w_0 \otimes w_1 \otimes (h_1(v_0^\ast \otimes v_0) \otimes h_2(v_0^\ast \otimes v_0) \\ 
  &&+ h_1(v_1^\ast \otimes v_1) \otimes h_2(v_1^\ast \otimes v_1)) \otimes \phi_2 \otimes w_3 \otimes \phi_4 \end{eqnarray*}
\item Apply $g$ to the first output factor of $h$.
\begin{eqnarray*} &&w_0 \otimes w_1 \otimes (g(h_1(v_0^\ast \otimes v_0)) \otimes h_2(v_0^\ast \otimes v_0) \\ &&+ g(h_1(v_1^\ast \otimes v_1)) \otimes h_2(v_1^\ast \otimes v_1)) \otimes \phi_2 \otimes w_3 \otimes \phi_4\end{eqnarray*}
\item Apply evaluation to the second output factor of $h$ and $\phi_2$.
\begin{eqnarray*} &&w_0 \otimes w_1 \otimes (g(h_1(v_0^\ast \otimes v_0)) \otimes \phi_2(h_2(v_0^\ast \otimes v_0)) \\&&+ g(h_1(v_1^\ast \otimes v_1)) \otimes \phi_2(h_2(v_1^\ast \otimes v_1)) \otimes w_3 \otimes \phi_4\end{eqnarray*}
\item Apply $f$ to $w_0 \otimes w_1 \otimes w_3$ .
\begin{eqnarray*}&&f(w_0 \otimes w_1 \otimes w_3) \otimes (g(h_1(v_0^\ast \otimes v_0)) \otimes \phi_2(h_2(v_0^\ast \otimes v_0)) \\&&+ g(h_1(v_1^\ast \otimes v_1)) \otimes \phi_2(h_2(v_1^\ast \otimes v_1)) \otimes \phi_4\end{eqnarray*}
\item Apply evaluation to the output  of $f$ and $\phi_4$, giving a product in $k$.
\begin{eqnarray*}&&\phi_4(f(w_0 \otimes w_1 \otimes w_3)) \cdot (g(h_1(v_0^\ast \otimes v_0)) \otimes \phi_2(h_2(v_0^\ast \otimes v_0)) \\&&+ g(h_1(v_1^\ast \otimes v_1)) \cdot \phi_2(h_2(v_1^\ast \otimes v_1))\end{eqnarray*}
\end{enumerate}

We made implicit use of the isomorphisms $V \otimes k \simeq V$ and $k \otimes k \simeq k$ in the calculation above. From commutativity of $k$ and the relation in figure 1.11 any way of making this computation will yield the same result, i.e. the diagram gives a well defined morphism.

From $\mathcal{P}(\text{Rep}_k(G,V))$ we now want to build a category equivalent to $\text{Rep}_k(G,V)$. 

\begin{mydef}
Let $\mathcal{C}(\mathcal{P}(\text{Rep}_k(G,V)))$ be the category whose objects are sequences $ \sigma \in \{+,-\}^n$, and where the morphism space $Hom(\alpha,\beta)$ is given by the $\alpha \cdot \beta^\ast$ box space of $\mathcal{P}(\text{Rep}_k(G,V))$. The $\otimes$ operation and composition in this category come from the action of the operad of OPTs. 
\end{mydef}

\begin{prop}
Define a functor $F:\mathcal{C}(\mathcal{P}(\text{Rep}_k(G,V))) \rightarrow \text{Rep}_k(G,V)$ on objects by $F(\sigma)=V^\sigma$  with $F(())=\mathbbm{1}$. Define $F$ on morphisms by $F(d)=f \in Hom(V^{\gamma},\mathbbm{1})$, where $\gamma$ is the index of the box space containing $d$, and $f$ is the morphism interpreted by the diagram $d$. $F$ is an equivalence of tensor categories.
\end{prop}
\begin{proof}
Follows from the fact that in our category $Hom(A,B) \simeq Hom(A \otimes B^\ast, \mathbbm{1})$ naturally in both arguments.
\end{proof}

 In all examples discussed in this thesis $V\simeq V^{\ast}$, so we fix an isomorphism $\varphi:V \rightarrow V^{\ast}$ and include a vertex $\varphi$ of degree $(+,+)$. Using $\varphi$ along with evaluation we can build the isomorphism $h:\text{Hom}(V^{\otimes n},V^{\otimes m}) \rightarrow \text{Hom}(V^{\otimes (n+m)}, \mathbbm{1})$, allowing us to consider only the box spaces $\beta_{(+,+,\dots,+)}$ of all positive index. We define $\beta_\smallblacksquare(n):=\beta_{(+,+,\dots,+)}$ where the index is a tuple of length $n$.
\begin{center}
\begin{tikzpicture}[x=0.75cm,y=0.75cm]
\clip (-2.8,-.1) rectangle (16,3.4);
\begin{scope}[decoration={
    markings,
    mark=at position 0.5 with {\arrow[scale=1.25]{latex}}}
    ]

\draw (2.25,1) rectangle (3.75,1.75);
\draw (3, 1.325) node {$f$};

\draw (2.25,1.375) node[anchor=east] {$\text{Hom}(V^{\otimes n}, V^{\otimes m}) \ni $};

\draw [|->,>=Latex] (4.25,1.375) to (6.25,1.375);

\draw (5.25,1.65) node {$h$};

\draw (8.75,1) rectangle (10.25,1.75);

\draw (9.5, 1.325) node {$f$};

\draw (6.75,1.125) rectangle (7.25,1.625);
\draw (7.75,1.125) rectangle (8.25,1.625);

\draw (7,1.35) node {$\varphi$};
\draw (8,1.35) node {$\varphi$};

\draw [-,postaction={decorate}] (9,1.75) to [controls=+(90:0.7) and +(90:0.7)] (8,1.625);
\draw [-,postaction={decorate}] (10,1.75) to [controls=+(90:1.7) and +(90:1.7)] (7,1.625);

\foreach \x in {2.5,3.5,9,10}
	\draw [-,postaction={decorate}] (\x,0) to (\x,1);
	
\foreach \x in {7,8}
	\draw [-,postaction={decorate}] (\x,0) to (\x,1.125);

\foreach \x in {2.5,3.5}
	\draw [-,postaction={decorate}] (\x,1.75) to (\x,2.75);

\draw (10.25,1.375) node[anchor=west] {$ \in \text{Hom}(V^{\otimes (n+m)},\mathbbm{1})$};


\end{scope}
\end{tikzpicture}
\end{center}

For a more in depth and general discussion of graphical languages and the relevant theorems and proofs, see \cite{graphical} along with the references to Joyal and Streets works within. This is all to say that we can understand $\text{Rep}_k(G,V)$ through the planar algebra $\mathcal{P}(\text{Rep}_k(G,V))$, and giving a presentation of $\mathcal{P}(\text{Rep}_k(G,V))$ is equivalent to a presentation of the category $\text{Rep}_k(G,V)$.

\begin{remark}
More generally the construction of this section works in any symmetric tensor category (see \cite{TensorCats} for definitions related to tensor categories) where $Hom(A,B) \simeq Hom(A \otimes B^\ast, \mathbbm{1})$, but in the scope of this thesis we only use categories of the form $Rep_k(G,V)$.
\end{remark}

\subsection{OSPAs generated by a set of vertices}
\label{sec:ospaverts}
Here we will discuss what it means to give a presentation of a planar algebra by generators and relations. We first introduce the planar algebra freely generated by certain elements (which we conventionally call vertices). We then describe what it means to take a quotient of a planar algebra by certain relations.

\begin{mydef} A \textbf{vertex} consists of a labelled disc, and a sequence of points along the boundary of that disc called the \textbf{degree}. In the context of oriented planar algebras and tangles we require the degree to be a sequence of oriented points. We mark the boundary interval between the first and last term of the degree with a $\star$.
\label{def:verts}
\end{mydef}

\begin{figure}[H]
\centering
\includegraphics{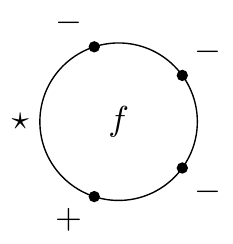}
\caption{An oriented vertex $f$ of degree $(+,-,-,-)$}
\end{figure}

\begin{mydef}
Let $S$ be a set of vertices. We say a vertex is \textbf{compatible} with an input disk of an OSPT when the degree is equal to the type of the input disc, i.e. a vertex and input disc are compatible when the sequences of oriented points along their boundaries are the same. The OSPA freely generated by $S$ is the collection of linear combinations of OSPTs where the input discs have been labelled with compatible vertices. The operad of OSPTs acts on such diagrams through the action on the underlying tangles.
\end{mydef}

Suppose for example we take the vertex set $G=\{f_{(+,+,+,-)},g_{(-)}, h_{(+,-,+,-)} \}$, denoting each vertex by its label with its degree written in the subscript. A pair of equivalent elements of the OSPA freely generated by $G$ are drawn in the figure below.

\begin{figure}[H]
\centering
\vinclude{optaction} $\mathlarger{\doteq}$ \vinclude{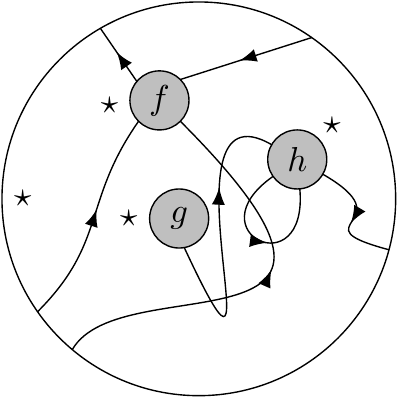}

\caption{Two elements of the OSPA generated by $G$, equivalent via symmetric isotopy.}
\end{figure}

We may have relations we want to impose on an OSPA, and to do this we define a planar ideal and the quotient of an OSPA by a planar ideal.

\begin{mydef}
\label{def:quotient}
A \textbf{planar ideal} $E$ in an (OS)PA $P$ is a collection of vector subspaces of the box spaces of $P$ closed under the action of the operad of (OS)PTs, i.e. if any input disc of some (OS)PT is labelled by an element of $E$ and the others are labelled by elements of $P$ then the result is in $E$. The \textbf{quotient} of an (OS)PA $G$ by $E$ is the collection of equivalence classes modulo $E$, and is denoted $G/E$.
\end{mydef}

A planar algebra which is a quotient of the free planar algebra generated by some set of vertices by a planar ideal will be called a \textbf{diagrammatic planar algebra}. Our goal is to give such a diagrammatic presentation of $\text{Alg}_\smallblacksquare = \mathcal{P}(\text{Rep}_k(G,V))$, i.e. to give an isomorphism of planar algebras $T:\text{Diag}_{\smallblacksquare} \rightarrow \text{Alg}_\smallblacksquare$ for some diagrammatic planar algebra $\text{Diag}_{\smallblacksquare}$.

\begin{remark}
To simplify diagrams and notation, instead of using labelled discs for vertices we will be using symbols like \textbf{[} or $\bullet$.
\end{remark}





\subsubsection{DTL$(2)$}
\label{sec:dtl2}
As an example of an OSPA of the form $G/E$ we look at Disoriented Temperley-Lieb at the value $2$. For simplicity, we use the symbols \textbf{]} and \textbf{[} for our generating vertices in place of labelled discs. We assume $\star$ is to the left of these symbols in the images below.

\begin{mydef} Define Disoriented Temperley-Lieb at the value $2$, denoted $DTL(2)$, to be the OSPA quotient $G/E$ for the generating set of vertices $G_i$ and relations $E_i$ below.

\hspace*{\fill}

\hspace{40mm} $G_1:$\vinclude{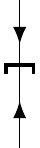} \hfill $G_2:$\vinclude{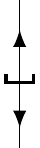} \hspace{40mm}

\hspace*{\fill}

\hspace*{\fill}

\hspace{5mm} \vinclude{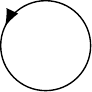} \hspace{0.05mm} $\overset{\scriptscriptstyle{E_1}}{=} 2$ \hfill \vinclude{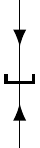} $\overset{\scriptscriptstyle{E_2}}{=} -$\vinclude{bracketmap} \hfill \vinclude{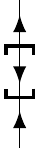} $\overset{\scriptscriptstyle{E_3}}{=}$ \vinclude{idmap} \hfill \vinclude{crossup} $\overset{\scriptscriptstyle{E_4}}{=}$ \vinclude{idmap} \hspace{1.5mm} \vinclude{idmap} $+$ \vinclude{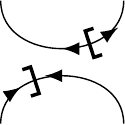} \hspace{5mm}

\hspace*{\fill}

\end{mydef}
 Let $G=SL_2(\mathbb{C}),k=\mathbb{C},$ and let $V$ be the standard two-dimensional representation of $G$ with basis $v_0=(1,0)$ and $v_1=(0,1)$. It is well known (with proof and exposition in \cite{SymmBook}) that all maps between $\otimes$ products of $V$ and $V^\ast$ are given by some combination of identity maps, evaluation, coevaluation, and the determinant. $V$ is self-dual, and we specify the isomorphism $\varphi: V \rightarrow V^{\ast}$ defined by $\varphi(v_0)=v_1^\ast, \varphi(v_1)=-v_0^\ast$, which also lets us express the determinant map $det:V \otimes V \rightarrow \mathbb{C}$ since $det = \epsilon \circ (\varphi \otimes 1_V)$.

\begin{prop}\label{prop:dtl2} There is a map of planar algebras $$T:DTL(2) \rightarrow \mathcal{P}(\text{Rep}_{\mathbb{C}}(SL_2(\mathbb{C}),V))$$ defined by $T(G_1)=\varphi,T(G_2)=\varphi^{-1}$.
\end{prop}
\begin{proof}
To show $T$ is a well-defined map we need to show the $E_i$ all hold in the image of $T$.

\begin{enumerate}[$E_1$:]
\item $\epsilon \cdot \eta(1) = \epsilon(v_1^{\ast} \otimes v_1 + v_0^{\ast} \otimes v_0) = 2$
\item This follows from rotation of $G_1$ and computing its value on a basis:

\begin{center}
\begin{tikzpicture}
\clip (3,0) rectangle (6.5,2);

\begin{scope}[decoration={
	markings,
	mark=at position 0.5 with {\arrow[scale=1.25,thick]{]}},
	mark=at position 0.8 with {\arrow[scale=1.25,>=latex]{<}},
	mark=at position 0.25 with {\arrow[scale=1.25,>=latex]{>}}}
	]

\draw [-,postaction={decorate}] (3.5,0.25) to (3.5,1.75);

\draw (3,0.25) to [controls= +(270:0.3) and +(-90:0.3)] (3.5,0.25);
\draw (3.5,1.75) to [controls= +(90:0.3) and +(90:0.3)] (4,1.75);

\draw (3,2) to (3,0.25);
\draw (4,1.75) to (4,0);

\draw (5,1) node {$=$};

\draw [-,postaction={decorate}] (6,2) to (6,0);

\end{scope}
\end{tikzpicture}
\end{center}
\begin{align*}
(&1_{V^{\ast}}\otimes \epsilon)(1_{V^\ast} \otimes \varphi \otimes 1_V)(\eta \otimes 1_V)(v_0) & &= \\
(&1_{V^{\ast}}\otimes \epsilon)(1_{V^\ast} \otimes \varphi \otimes 1_V)(v_0^\ast \otimes v_0 \otimes v_0 + v_1^\ast \otimes v_1 \otimes v_0) & &= \\
(&1_{V^{\ast}}\otimes \epsilon)(v_0^\ast \otimes v_1^\ast \otimes v_0 - v_1^\ast \otimes v_0^\ast \otimes v_0) & &= -v_1^\ast \\
\end{align*}
\begin{align*}
(&1_{V^{\ast}}\otimes \epsilon)(1_{V^\ast} \otimes \varphi \otimes 1_V)(\eta \otimes 1_V)(v_1) & &= \\
(&1_{V^{\ast}}\otimes \epsilon)(1_{V^\ast} \otimes \varphi \otimes 1_V)(v_0^\ast \otimes v_0 \otimes v_1 + v_1^\ast \otimes v_1 \otimes v_1) & &= \\
(&1_{V^{\ast}}\otimes \epsilon)(v_0^\ast \otimes v_1^\ast \otimes v_1 - v_1^\ast \otimes v_0^\ast \otimes v_1) & &= v_0^\ast\\
\end{align*}
\item $\varphi \cdot \varphi^{-1} = 1_{V^{\ast}}$
\item The non-identity term on the right hand side takes the values $v_0 \otimes v_0 \mapsto 0, v_1 \otimes v_1 \mapsto 0, v_0 \otimes v_1 \mapsto v_1 \otimes v_0 - v_0 \otimes v_1, v_1 \otimes v_0 \mapsto v_0 \otimes v_1 - v_1 \otimes v_0$, and adding the identity map gives us the left hand side.
\end{enumerate}

Now to show that $T$ is an isomorphism, we define for each $n \in \mathbb{N}$ a sequence $X_n$ in the $n$-box of $DTL(2)$ and then show that:

\begin{enumerate}
\item The sequence $T(X_n)$ is a basis for $Hom(V^{\otimes n}, \mathbbm{1})$, showing surjectivity of $T$.
\item The sequence $X_n$ spans the $n$-box of $DTL(2)$, showing injectivity of $T$.
\end{enumerate}

Take $X_n$ to be the set of diagrams built from nonintersecting caps with a single left-facing bracket per cap (so $X_n$ is empty when $n$ is odd).

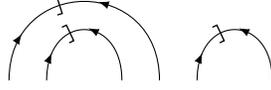
\begin{figure}[H]
\centering
\begin{tikzpicture}

\begin{scope}[decoration={
	markings,
	mark=at position 0.2 with {\arrow[scale=1,>=latex]{>}},
	mark=at position 0.6 with {\arrow[scale=1,>=latex]{<}},
	mark=at position 0.4 with {\arrow[scale=1.25,>=latex]{]}}
}
	]
\draw[-,postaction=decorate] (0,0) to [controls=+(90:1.4) and +(90:1.4)] (2,0);
\draw[-,postaction=decorate] (0.5,0) to [controls=+(90:0.9) and +(90:0.9)] (1.5,0);
\draw[-,postaction=decorate] (2.5,0) to [controls=+(90:0.9) and +(90:0.9)] (3.5,0);
\end{scope}
\end{tikzpicture}
\caption{An example diagram from $X_6$}
\end{figure}

These are indexed by matching parentheses, so that $dim(X_{2k}) = C_k$, where $C_k$ is the Catalan number $\frac{1}{k+1} {{2k} \choose {k}}$. This is equal to $dim(Hom(V^{\otimes 2k},\mathbbm{1}))$, since we have one irreducible representation of $SL_2(\mathbb{C})$ up to isomorphism of each dimension, so that $dim(Hom(V^{\otimes 2k},\mathbbm{1}))$ is equal to the multiplicity of $\mathbbm{1}$ in $V^{\otimes 2k}$. Further it is well known that $V \otimes V_i \simeq V_{i-1} \oplus V_{i+1}$ where $V_n$ is the irreducible of dimension $n$, so it follows that the multiplicity of $\mathbbm{1}$ in $V^{\otimes 2k}$ is $C_k$. Then, if we show $T(X_n)$ is independent it forms a basis. The proof of this is in Proposition $10$, and uses some techniques discussed later so we defer for now.

To show $X_n$ spans the $n$-box of $DTL(2)$, we note that in an arbitrary diagram we can first remove all crossings with $E_4$. Then remove all components not connected to the ground using $E_1,E_2,$ and $E_3$, which can not introduce any new crossings. Finally using $E_2$ and $E_3$ we reduce to a single bracket on each strand component, and direct it leftward.

\end{proof}



We have a special element in $DTL(2)$, the symmetrizer on $n$ strands, that we will refer to later. This is a special case of the (disoriented version) of the Jones-Wenzl projections \cite{wenzl} when the circle value is $2$.

\begin{mydef}
The \textbf{symmetrizer} or \textbf{Disoriented Jones-Wenzl} projection on $n$ strands is the element of $DTL(2)$ defined by 
$$DJW_n=Sym_n=\frac{1}{n!} \sum_{\sigma \in S_n} \includegraphics[scale=1,valign=c]{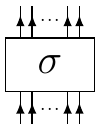}$$ 
where $\sigma$ connects incoming strand $i$ to outgoing strand $\sigma(i)$.
\end{mydef}

\subsection{Skein Theory for Planar Algebras}
Skein theory was first introduced by Conway in studying knot theory \cite{Cskein}. This lead to computation of invariants such as the Jones polynomial \cite{JPoly}, and is a useful technique in studying planar algebras. A skein theory for a planar algebra is defined by giving a set of directed local relations which allow simplification of a diagram (i.e. we give a notion of what a simplified diagram is, and the relations have a side identified as being simpler). It's natural to look for bases of the box spaces of a planar algebra, and skein theory gives a technique for rewriting a diagram in terms of some basis of diagrams. For example consider $DTL(2)$ where we have the relation below and can use it to simplify a diagram by removing all crossings (all terms resulting in the application of the crossing relation have $1$ fewer crossing, and we can repeat until there are none).

\begin{center}
\vspace{2mm}
\vinclude{crossup} $=$ \vinclude{idmap} \hspace{2mm} \vinclude{idmap} $+$ \vinclude{capcup}
\vspace{2mm}
\end{center}

We illustrate with the example below. The tree shows the terms which appear during the reduction of our initial diagram (up to sign as brackets and orientations have been omitted for simplicity). We see that no terms in the result have a strand crossing.

\begin{center}
\includegraphics[scale=0.85]{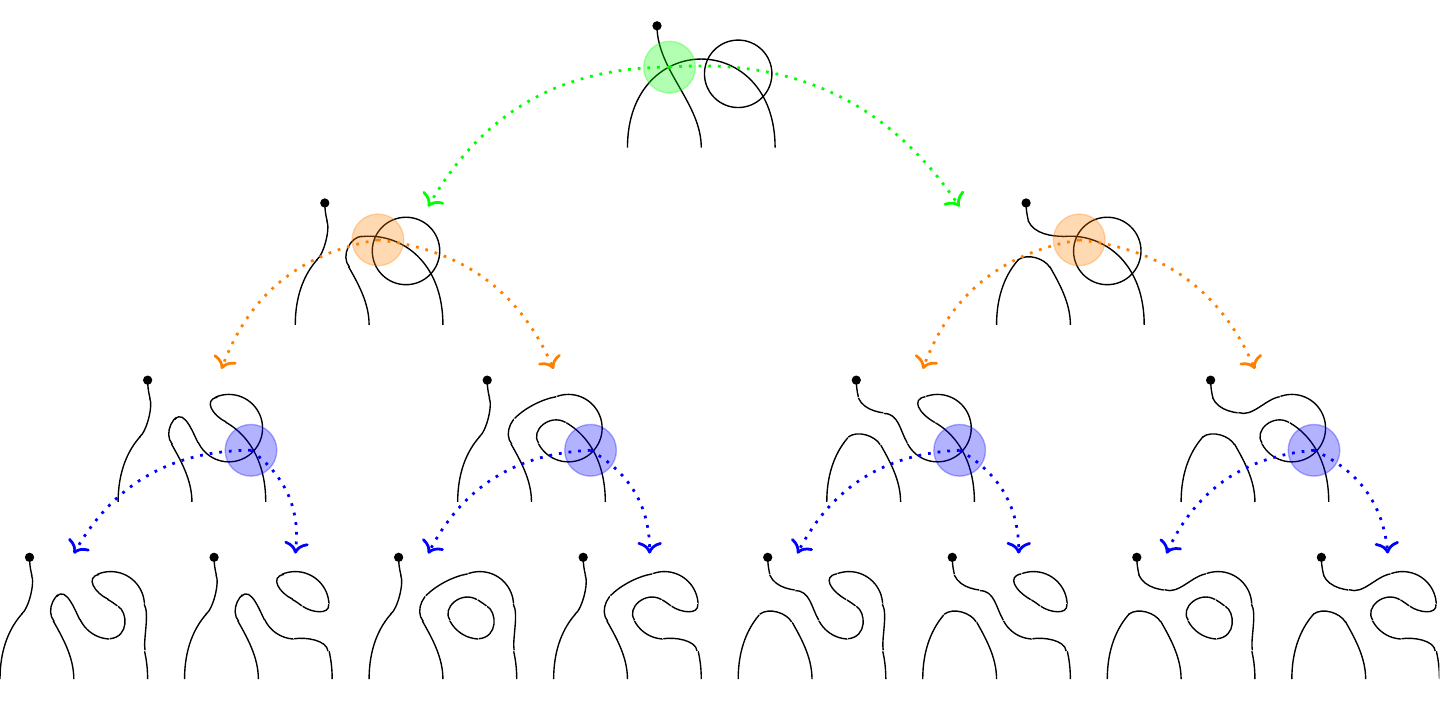}
\end{center}

When defining a skein theory, we need to come up with a full set of relations as well as a notion of simplification so that we can reduce an arbitrary picture. We look for a basis of diagrams for our box spaces that are simple by some definition, and show that our skein relations give a simplification algorithm that ends in terms of the chosen basis. The presentations in Section \ref{sec:presentations} define the skein theories for the planar algebras discussed in this paper. In Sections \ref{sec:yn_span_0} and \ref{sec:yn_span_p} we make use of skein theory to reduce arbitrary diagrams and prove that we have bases for our box spaces.

\subsection{Light Leaf and Jellyfish type bases and reduction arguments}
We give light leaf and jellyfish style arguments, combining the two in the characteristic $p$ case. Jellyfish style arguments are as seen in \cite{BMPS,Jelly1,Jelly2} and light leaf arguments are as seen in \cite{LL,elias-will,Skein,elias-hweights}. In either style of argument our goal is to find bases for the box spaces of some planar algebra.

In a light leaf style argument, we construct a sequence of maps indexed by the fusion graph for some indecomposable $V$. In characteristic $0$ for representations of a Lie group, these maps are constructed in a way that we can pair the maps with basis vectors of $V^{\otimes n}$, and then use an argument involving highest weights to show the maps are linearly independent. Finally we show that we can reduce an arbitrary diagram to the span of the maps that we have constructed.

In a jellyfish style argument, we define a sequence of maps using some generating set of vertices such that the maps are as simple as possible with respect to a reduction algorithm. The reduction algorithm uses a set of directed equalities (we define one side to be simpler than the other) to start from an arbitrary diagram and reduce to the jellyfish diagrams. We need to then show the jellyfish diagrams are independent so that we have bases for the box spaces.

In our characteristic $p$ example, we use a light leaf style argument to get a basis of the box spaces of $Alg_p$, so that our map $T_p$ is surjective. Notably the typical argument using weight vectors doesn't quite work, since there is no basis which behaves well with respect to weights, but the same technique works via a modified argument. We then use a jellyfish style argument to show $T_p$ is injective, by giving a reduction algorithm and a bijection between the light leaf and jellyfish bases.

\subsection{Outline of technique for finding a diagrammatic presentation of $\text{Rep}_k(G,V)$.}

Our goal is to determine a generating set of vertices and relations for $\text{Diag}_{\smallblacksquare}$ so that $\text{Diag}_{\smallblacksquare} \simeq \text{Alg}_\smallblacksquare$ as planar algebras. We use the following procedure.
 \begin{enumerate}
 \item Construct a set of vertices and relations for $\text{Diag}_{\smallblacksquare}$, and a map of planar algebras $T:\text{Diag}_{\smallblacksquare} \rightarrow \text{Alg}_\smallblacksquare$. To define $T$ we need to pick a morphism in $\text{Alg}_\smallblacksquare$ for every vertex, and check that the images of all the relations hold in $\text{Alg}_\smallblacksquare$.
 \item Construct a collection of sequences $X_\smallblacksquare$ indexed by $\mathbb{N}$ where the sequence of index $n$ has terms in $\text{Diag}_{\smallblacksquare}(n)$. We denote the $n^{\text{th}}$ sequence by $X_\smallblacksquare(n)$, and require that $X_\smallblacksquare(n)$ has length $\dim(\text{Alg}_\smallblacksquare(n))$ and $T(X_\smallblacksquare(n))$ is linearly independent.
 \item Construct a collection of sequences $Y_\smallblacksquare$ indexed by $\mathbb{N}$ where the sequence of index $n$ has terms in $\text{Diag}_{\smallblacksquare}(n)$. We require that $Y_\smallblacksquare(n)$ has length $\dim(\text{Alg}_\smallblacksquare(n))$ and spans $\text{Diag}_{\smallblacksquare}(n)$.

\end{enumerate}

 Now take $X_\smallblacksquare(n)$ and $Y_\smallblacksquare(n)$ to be the sequences in the statement of Lemma 1. The result of Lemma 1 implies that $\text{Diag}_{\smallblacksquare} \simeq \text{Alg}_\smallblacksquare$ as planar algebras, since the restriction $\left.T \right|_{\text{Diag}_{\smallblacksquare}(n)}$ is an isomorphism for each $n$. Further, $X_\smallblacksquare(n)$ and $Y_\smallblacksquare(n)$ are bases of $text{Diag}_{\smallblacksquare}(n)$ which we will call the light leaf basis and jellyfish basis respectively.

\subsection{Disorientation markings}
In our main results we use Disoriented Temperley Lieb as our base planar algebra, however it is more common to see Temperley Lieb presented as an unoriented planar algebra. This is because in the semisimple case there is a change of pivotal structure allowing us to disregard strand orientation, while in the examples we are interested in that change of pivotal structure no longer exists. 

We recall that a pivotal structure on a rigid monoidal category is a natural isomorphism of monoidal functors $p:Id \rightarrow (-)^{\ast\ast}$ between the identity and double dual. A change of pivotal structure is a natural automorphism of the identity functor, and we form the group $\Delta$ of changes of pivotal structure. We will be in the context of a planar algebra $\otimes$-generated by an object $V$, where $V$ is in some pivotal tensor category. Since any $\delta\in \Delta$ is monoidal, $\delta$ is determined by a choice of automorphism of $V$. For example taking $f \in Aut(V)$ we get the components $\delta_{V^{\otimes n}}=f^{\otimes n}$.

For Temperley-Lieb, the change of pivotal structure is uniquely determined using the tensor product taking the negative identity map on $V$ as our automorphism. This allows us to ignore strand orientation. In our main results we have a vertex $\bullet$ of degree $(+)$, and the change of pivotal structure used for Temperley-Lieb is not natural with respect to $\bullet$. Further, any change of pivotal structure would be determined by some choice of $\alpha$ and $\beta$ in:
$$\alpha \vinclude{idmap} + \beta \vinclude{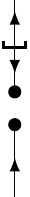}$$

 Since $\delta$ must be natural (in particular natural with respect to $\bullet$), applying $\bullet$ to the top of both sides of the equation as in the computation below shows $\alpha=1$, so that it is not possible to just remove the disorientation markings. The term with coefficient $\beta$ below vanishes by relation $E_2$ in Theorem \ref{thm:char0}.
$$\vinclude{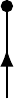} = \alpha \vinclude{dotmap} + \beta \vinclude{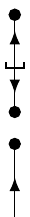} = \alpha \vinclude{dotmap} \implies \alpha = 1 $$

\section{$\text{URep}_{\mathbb{C}}(\mathbb{C}^+)$}
\label{ch:char0}
We want to consider representations of the group $\mathbb{C}^+$. By the correspondence with representations of Lie algebras, these are all given by $z \mapsto \exp(zM)$, determined by a choice of $M \in \mathfrak{gl}(\mathbb{C}^n)$. Since $\exp: \mathfrak{gl}(\mathbb{C}^n) \rightarrow GL(\mathbb{C}^n)$ is surjective, each choice of $A \in GL(\mathbb{C}^n)$ gives a representation determined by $1 \mapsto A$. Representations are then conjugacy classes of invertible matrices, and we pick a representative in Jordan normal form. This makes it clear that indecomposable representations correspond to Jordan blocks. We restrict to unipotent representations, so consider the tensor subcategory of $Rep_\mathbb{C}(\mathbb{C}^+)$ where all eigenvalues of the defining matrix of the representation are equal to $1$.

Let $V_n=(\mathbb{C}^n,\phi_n)$ where $\phi_n: \mathbb{C}^+ \rightarrow GL(\mathbb{C}^n)$ is defined by $\phi_n(1)=J_n$, and $J_n$ is the Jordan block of dimension $n$ with eigenvalue $1$. In this section we set $V=V_2$ and study $Alg_0=\mathcal{P}(Rep_{\mathbb{C}}(\mathbb{C}^+,V))$. Taking the standard basis of $\mathbb{C}^2$, $v_0 = (1,0)$ and  $v_1 = (0,1)$, we use the isomorphism $\varphi:V \rightarrow V^{\ast}$ defined by $\varphi(v_0)=v_1^{\ast}$, $\varphi(v_1)=-v_0^{\ast}$. We will construct a diagrammatic planar algebra $Diag_0$ and an isomorphism of planar algebras to $Alg_0$ 

\subsection{$\text{Hom}_{\mathbb{C}^+}(V_n,V_m)$}\label{homs}
We need to find the $\mathbb{C}$-linear maps $T:\mathbb{C}^n \rightarrow \mathbb{C}^m$ satisfying $T \cdot J_n = J_m \cdot T$. Writing $J_i = I_i + N_i$ where $I_i$ is the $i$-dimensional identity map, we have $T \cdot (I_n + N_n) = (I_m + N_m) \cdot T$, so it is sufficient to find all $T$ such that $T \cdot N_n = N_m \cdot T$. Consider an arbitrary $T=\left(a_{i,j}\right) \in Hom(V_n,V_m)$. The matrix $N_n$ acts on $T$ on the right by shifting all columns to the right once, and replacing the first column by the zero vector. Similarly, $N_m$ acts on $T$ on the left by shifting all rows up once, and replacing the last row by the zero vector. This gives the relation $a_{i,j} = a_{i+1,j+1}$ on the coordinates of T. Since further $a_{i,1}=0$ for $i > 1$ and $a_{m,j} = 0$ for $j \geq 1$, we know all entries of $T$ below the main diagonal (the set of entries $\{a_{i,i}\}$) must be $0$, and the main diagonal will be zero when $n > m$. Every diagonal of T above the main diagonal (and including the main diagonal when $n \leq m$) will then correspond to one free parameter. We get $dim(Hom(V_n,V_m))=\min(n,m)$.

To give an explicit basis of $Hom(V_n,V_m)$, let $s = \min(n,m)$. We fix the standard ordered basis for both $\mathbb{C}^n$ and $\mathbb{C}^m$, and consider $\mathbb{C}^s$ to be a subspace of each of $\mathbb{C}^n$ and $\mathbb{C}^m$ by taking the first $s$ basis vectors of $\mathbb{C}^m$, and the last $s$ basis vectors of $\mathbb{C}^n$. The set of maps $\{T_i\}_{0 \leq i \leq s-1}$, where $T_i$ acts by $N_s^i$ on $\mathbb{C}^s \subseteq \mathbb{C}^n$ and $0$ on the complement of $\mathbb{C}^s$, forms a basis of $Hom(V_n,V_m)$. For example: 

\noindent$Hom(V_3,V_4)$ has basis $$(T_0, T_1, T_2) = \left( \begin{bmatrix}
 1 & 0 & 0 \\
 0 & 1 & 0 \\
 0 & 0 & 1 \\
 0 & 0 & 0 
\end{bmatrix}, \begin{bmatrix}
 0 & 1 & 0 \\
 0 & 0 & 1 \\
0 & 0 & 0 \\
 0 & 0 & 0 
\end{bmatrix}, \begin{bmatrix}
 0 & 0 & 1 \\
 0 & 0 & 0 \\
 0 & 0 & 0 \\
 0 & 0 & 0
\end{bmatrix} \right).$$

\noindent$Hom(V_4,V_3)$ has basis $$(T_0,T_1,T_2) = \left( \begin{bmatrix}
 0 & 1 & 0 & 0 \\
 0 & 0 & 1 & 0 \\
 0 & 0 & 0 & 1 \\
\end{bmatrix},\begin{bmatrix}
 0 & 0 & 1 & 0 \\
 0 & 0 & 0 & 1 \\
 0 & 0 & 0 & 0 \\
\end{bmatrix},\begin{bmatrix}
 0 & 0 & 0 & 1 \\
 0 & 0 & 0 & 0 \\
 0 & 0 & 0 & 0 \\
\end{bmatrix} \right).$$

\subsection{Defining $Diag_0$ and the map to $Alg_0$}
Define $Diag_0$ to be the OSPA defined by the generators and relations below. Note that $Diag_0$ is defined by the generators and relations of $DTL(2)$, along with the added generator $G_1$ and generating relation $E_2$.
\begin{center}
\begin{tikzpicture}[scale=0.9]
\begin{scope}[decoration={
    markings,
    mark=at position 0.5 with {\arrow[scale=1.25]{latex}}}
    ]
\clip (-1,-1.25) rectangle (12.25,3.5);

\draw [postaction={decorate}] (0,0) circle (17pt);
\draw (1.2,0) node[scale=1.4] {$=$};
\draw (1.2,.05) node[anchor=south,scale=0.8] {$E_3$};
\draw (1.5,.05) node[anchor=west,scale=1.4] {$2$};
\draw [dotted] (-0.9,-1) rectangle (2.25,1);

\draw [-,postaction=decorate] (9.25,-0.75) to (9.25,0.75);
\draw [-,postaction=decorate] (9.75,-0.75) to (9.75,0.75);
\draw (10.25,0) node[scale=1.2] {$+$};

\draw [-,postaction={decorate}] (0,1.5) to (0,2.75);
\draw [fill=black] (0,2.75) circle (1.5pt);

\draw [-,postaction={decorate}] (8,3.25) to (8,1.5);
\draw (7,2.25) node[scale=1.4] {$=$};
\draw (7,2.3) node[anchor=south,scale=0.8] {$E_1$};

\draw [dotted] (5.5,1.25) rectangle (8.5,3.5);
\end{scope}

\draw (-0.1,2.25) node[anchor=east] {$G_1:$};

\begin{scope}[decoration={
	markings,
	mark=at position 0.5 with {\arrow[scale=1.25,thick]{[}},
	mark=at position 0.25 with {\arrow[scale=1.25]{latex}},
	mark=at position 0.8 with {\arrow[scale=1.25,>=latex]{<}}}
	]

\draw [-,postaction={decorate}] (2,3.25) to (2,1.5);
\draw [-,postaction={decorate}] (3.5,-0.75) to (3.5,0.75);
\draw [-,postaction={decorate}] (5.5,0.75) to (5.5,-0.75);

\draw (4.5,0) node[scale=1.4] {$=$};
\draw (4.5,0.05) node[scale=0.8,anchor=south] {$E_4$};
\draw (5,0) node[scale=1.2] {$-$};
\draw [dotted] (3,-1) rectangle (6,1);

\draw [dotted] (6.5,-1) rectangle (12.25,1);
\end{scope}

\draw (1.9,2.25) node[anchor=east] {$G_2:$};

\begin{scope}[decoration={
	markings,
	mark=at position 0.5 with {\arrow[scale=1.25,thick]{]}},
	mark=at position 0.8 with {\arrow[scale=1.25]{latex}},
	mark=at position 0.25 with {\arrow[scale=1.25,>=latex]{<}}}
	]
\draw [-,postaction={decorate}] (4,3.25) to (4,1.5);

\draw [-,postaction={decorate}] (10,3.25) to (10,1.5);
\draw [fill=black] (10,3.25) circle (1.5pt);
\draw [fill=black] (10,1.5) circle (1.5pt);
\draw (11,2.25) node[scale=1.4] {$=$};
\draw (11,2.3) node[anchor=south,scale=0.8] {$E_2$};
\draw (11.5,2.3) node[scale=1.3,anchor=west] {$0$};
\draw [dotted] (9.5,1.25) rectangle (12.25,3.5);

\end{scope}

\begin{scope}[decoration={markings,
	mark=at position 0.35 with {\arrow[scale=1.25,thick]{]}},
	mark=at position 0.55 with {\arrow[scale=1.25,>=latex]{>}},
	mark=at position 0.2 with {\arrow[scale=1.25,>=latex]{<}}}
	]
\draw [-,postaction=decorate] (12,.75) to [controls= +(-90:0.75) and +(-90:0.75)] (10.75,0.7);
\end{scope}
\begin{scope}[decoration={markings,
	mark=at position 0.35 with {\arrow[scale=1.25,thick]{]}},
	mark=at position 0.55 with {\arrow[scale=1.25,>=latex]{<}},
	mark=at position 0.2 with {\arrow[scale=1.25,>=latex]{>}}}
	]
\draw [-,postaction=decorate] (10.75,-0.75) to [controls= +(90:0.75) and +(90:0.75)] (12,-0.75);
\end{scope}

\begin{scope}[decoration={
	markings,
	mark=at position 0.3 with {\arrow[scale=1.25]{latex}},
	mark=at position 0.8 with {\arrow[scale=1.25]{latex}}}
	]

\draw [-,postaction={decorate}] (6.75,-0.75) to (8,0.75);
\draw [-,postaction={decorate}] (8,-0.75) to (6.75,0.75);	

\end{scope}


\draw [-,postaction={decoration={markings,mark=at position 0.35 with {\arrow[scale=1.25,thick]{[}},mark=at position 0.65 with {\arrow[scale=1.25,thick]{]}},mark=at position 0.2 with {\arrow[scale=1.25,>=latex]{<}},mark=at position 0.8 with {\arrow[scale=1.25,>=latex]{<}},mark=at position 0.53 with {\arrow[scale=1.25,>=latex]{>}}},decorate}] (6,1.5) to (6,3.25);

\draw (3.9,2.25) node[anchor=east] {$G_3:$};

\draw (8.5,0) node[scale=1.4] {$=$};
\draw (8.5,0.05) node[scale=0.8,anchor=south] {$E_5$};
\end{tikzpicture}
\end{center}
\begin{prop}\label{prop:ismap0}
There is a map of planar algebras $T:Diag_0 \rightarrow Alg_0$ determined by the values $T(G_1)=v_1^{\ast},T(G_2)=\varphi$ and $T(G_3)=\varphi^{-1}$. 
\end{prop}

\begin{proof}
We need to show each of $E_1$ through $E_5$ hold in the image of $T$ so that this map is well defined. All but $E_2$ were shown in the proof of Proposition \ref{prop:dtl2} in Section \ref{sec:dtl2}. $E_2$ follows from the computation $v_1^{\ast} \cdot \varphi^{-1} \cdot v_1^{\ast\ast}(1)=v_1^{\ast} \cdot \varphi^{-1} (v_1^{\ast})=v_1^{\ast}(v_0)=0$.
\end{proof}

\subsection{Constructing $X_0(n)$ and showing independence of $T(X_0(n))$}
The indecomposable representations that appear in $\otimes$-powers of $V$ are enumerated by the sequence $V_n$.  We define the \textbf{fusion graph} for $V$, $\Gamma(V)$, to have a vertex for each indecomposable representation and an edge from $V_i$ to $V_j$ for each summand isomorphic to $V_j$ in $V \otimes V_i$. When $i \geq 2$ we have the rule $V \otimes V_i \simeq V_{i+1} \oplus V_{i-1}$ so that $\Gamma(V)$ is

\begin{tikzpicture}[scale=0.85]
\clip(-0.5,-1.85) rectangle (14,-.25);
\foreach \x in {1,...,7}
	\draw (2*\x-2,-1.1) node[anchor=north] {$V_{\x}$};
\foreach \x in {1,...,6}
	{\draw [->, >=latex] (2*\x-2,-1) to (2*\x-0.1,-1);
	\draw [<-, >=latex] (2*\x-1.9,-1) to (2*\x,-1);}
\draw (12,-1) to (13,-1);

\foreach \x in {0,...,6}
	\draw [fill=black] (2*\x,-1) circle (2pt);
\foreach \x in {1,...,3}
	\draw [fill=black] (13+.25*\x,-1) circle (0.5pt);

\draw [->,>=latex] (12.5,-1) to (12.1,-1);

\end{tikzpicture}

Note that this is the underlying graph for the automata $M_0$ of Section $2.2$, and since all states are accepting, paths of length $n$ on $\Gamma(V)$ are in bijection with $L_{M_0}(n)$.

\begin{prop}
$\#(L_{M_0}(n))=dim(Alg_0(n))$
\end{prop}

\begin{proof}
  We have $\dim(Alg_0(n))=\text{dimHom}(V^{\otimes n},\mathbbm{1})=\text{dimHom} \big( \sum{\alpha_i V_i,\mathbbm{1} \big) }=\sum{\alpha_i \text{dimHom}(V_i,\mathbbm{1})}$. As in Section \ref{homs}, dimHom$(V_j, \mathbbm{1}) = 1$ for any indecomposable $V_j$. We then have dim$(B_n)=\sum{\alpha_i}$, which is the number of summands of $V^{\otimes n}$. We can see summands of $V^{\otimes n}$ are in bijection with $L_{M_0}(n)$ by induction on $n$. Assume we have a direct sum decomposition of $V^{\otimes n}$ and a bijection between summands of $V^{\otimes n}$ and $L_{M_0}(n)$. By definition of $\Gamma$ the summands of $V^{\otimes (n+1)}$ will be the indecomposables that are adjacent to the summands of $V^{\otimes n}$, so append the adjacency edge to the path from the bijection at level $n$.
\end{proof}

We will now construct our sequence $X_n$ of length $\dim(Alg_0(n))$.

\begin{mydef}

To define a map $\delta_n :L_{M_0}(n) \rightarrow Diag_0(n)$ identify concatenation in a word $w$ with composition of diagrams, and identify each letter of the alphabet with a portion of a picture as below, where $\emptyset$ signifies the end of the word.
\begin{center}
\begin{tikzpicture}[x=.5cm,y=.5cm]
\clip(-1,-0.5) rectangle (16,5);
\begin{scope}[decoration={
    markings,
    mark=at position 0.5 with {\arrow[scale=1.25]{latex}}}
    ]

\draw [dotted] (1,0) rectangle (5,4);
\draw [-,postaction={decorate}] (1,2) to (5,2);
\draw [-,postaction={decorate}] (1,3) to (5,3);
\draw [fill=black] (3,2.25) circle (0.5pt);
\draw [fill=black] (3,2.50) circle (0.5pt);
\draw [fill=black] (3,2.75) circle (0.5pt);
\draw (3,4) node[anchor=south] {$R$};
\draw [-,postaction={decorate}] (3,0) to [controls = +(270:-1) and +(0:-1)] (5,1.5);

\draw [dotted] (6,0) rectangle (10,4);
\draw [-,postaction={decorate}] (6,2) to (10,2);
\draw [-,postaction={decorate}] (6,3) to (10,3);
\draw [fill=black] (8,2.25) circle (0.5pt);
\draw [fill=black] (8,2.5) circle (0.5pt);
\draw [fill=black] (8,2.75) circle (0.5pt);
\end{scope}
\begin{scope}[decoration={
    markings,
    mark=at position 0.5 with {\arrow[scale=1.25,thick]{]}},
    mark=at position 0.25 with {\arrow[scale=1.25]{latex}},
    mark=at position 0.75 with {\arrow[scale=1.25,>=latex]{<}}}
    ]

\draw [-,postaction={decorate}] (8,0) to [controls = +(270:-1) and +(0:1)] (6,1.5);
\end{scope}
\begin{scope}[decoration={
	markings,
	mark=at position 0.45 with {\arrow[scale=1.25]{latex}}}
	]
\draw (8,4) node[anchor=south] {$L$};

\draw [dotted] (11,0) rectangle (15,4);
\draw (13,4) node[anchor=south] {$\emptyset$};

\draw [-,postaction={decorate}] (11,3) to [bend right] (12,3.5);
\draw [-,postaction={decorate}] (11,2) to [bend right] (12.25,2.5);
\draw [fill=black] (12,3.5) circle (1.5pt);
\draw [fill=black] (12.25,2.5) circle (1.5pt);

\draw [fill=black] (11.5,2.25) circle (0.5pt);
\draw [fill=black] (11.5,2.5) circle (0.5pt);
\draw [fill=black] (11.5,2.75) circle (0.5pt);

\end{scope}
\end{tikzpicture}
\end{center}

 To define $\delta_n(w)$ replace each letter of $w$ with its identified picture, and then glue each picture end to end, including the picture for $\emptyset$ (glue dots onto any remaining strands).
\end{mydef}

For example, take $w=RRRLLRL$:

\begin{tikzpicture}[scale=0.9]
\clip (-1.51,-0.5) rectangle (16,1.7);

\draw [dotted] (0,0) rectangle (8,1.5);
\foreach \x in {1,...,7}
	\draw [dotted] (\x,0) to (\x,1.5);
\draw (-1.5,.75) node[anchor=west] {$\delta_7(w)=$};
\draw (8,.75) node[anchor=west] {$=$};
\foreach \x in {1,2,3,6}
	\draw (-0.5 + \x,0) node[anchor=north] {$R$};
\foreach \x in {4,5,7}
	\draw (-0.5 + \x,0) node[anchor=north] {$L$};

\draw (7.5,0) node[anchor=north] {$\emptyset$};

\begin{scope}[decoration={
    markings,
    mark=at position 0.5 with {\arrow{latex}}}
    ]

\draw [-,postaction={decorate}] (0.5,0) to [controls= +(90:0.5) and +(180:0.3)] (1,0.9);
\draw [-,postaction={decorate}] (1.5,0) to [controls= +(90:0.35) and +(180:0.3)] (2,0.7);
\draw [-,postaction={decorate}] (2.5,0) to [controls= +(90:0.2) and +(180:0.3)] (3,0.5);
\foreach \x in {1,...,6}
	\draw [-,postaction={decorate}] (\x,0.9) to (\x+1,0.9);

\draw [-,postaction={decorate}] (2,0.7) to (3,0.7);
\draw [-,postaction={decorate}] (3,0.7) to (4,0.7);

\draw [-,postaction={decorate}] (5.5,0) to [controls = +(90:0.2) and +(180:0.3)] (6,0.5);

\draw [-,postaction={decorate}] (7,0.9) to [bend right] (7.4,1.2);

\draw [fill=black] (7.4,1.2) circle (1.25pt);

\draw [-,postaction={decorate}] (8.85,0) to (8.85,0.6);
\draw [fill=black] (8.85,0.6) circle (1.25pt);

\end{scope}
\begin{scope}[decoration={
	markings,
	mark=at position 0.5 with {\arrow[thick]{]}},
	mark=at position 0.3 with {\arrow[>=latex]{>}},
	mark=at position 0.8 with {\arrow[>=latex]{<}}}
	]

\draw [-,postaction={decorate}] (6.5,0) to [controls= +(90:0.3) and +(0:0.3)] (6,0.5);
\draw [-,postaction={decorate}] (4.5,0) to [controls= +(90:0.4) and +(0:0.4)] (4,0.7);
\draw [-,postaction={decorate}] (3.5,0) to [controls= +(90:0.3) and +(0:0.3)] (3,0.5);

\end{scope}

\begin{scope}[decoration={
	markings,
	mark=at position 0.75 with {\arrow[thick]{]}},
	mark=at position 0.65 with {\arrow[>=latex]{>}},
	mark=at position 0.9 with {\arrow[>=latex]{<}}}
	]

\draw [-,postaction={decorate}] (10.25,0) to [controls = +(90:0.5) and +(90:0.5)] (9.5,0);
\draw [-,postaction={decorate}] (10.5,0) to [controls = +(90:0.9) and +(90:0.9)] (9.25,0);
\draw [-,postaction={decorate}] (11.5,0) to [controls = +(90:0.5) and +(90:0.5)] (10.75,0);

\end{scope}
\end{tikzpicture}

The images of $\delta_n$ form our sequence $X_n$.

\begin{prop}\label{prop:indep0}
The sequence $T(X_n)$ is linearly independent.
\end{prop}
\begin{proof}
As in Lemma \ref{indep}, let $S$ be $L_{M_0}(n)$ with lexicographic order taking $R>L$, and let $g$ be $T \circ \delta_n$. To define $f$ assign to each $x \in L_{M_0}$ a vector $f(x) \in V^{\otimes n}$ by identifying $R$ with \high, $L$ with \low, and concatenation with $\otimes$ (e.g. $f(RRLRLLRL)=$ \high\high\low\high\low\low\high\low)
By construction of the pairing we have $g(x)(f(x))=-1$, so condition $1$ of Lemma \ref{indep} holds. To show condition $2$ holds, note that if $x<y$ then
\begin{align*} 
f(y) &= \high^{i} \cdot \high \cdot u, \hspace{5mm} u \in V^{\otimes (n-(i+1))} \\
f(x) &= \high^{i} \cdot \low \cdot w, \hspace{5mm} w \in V^{\otimes (n-(i+1))}
\end{align*}
so that $g(x)(f(y))$ will have the form
\begin{center}
\begin{tikzpicture}[x=1cm,y=0.8cm]
\clip (-3,-1.2) rectangle (9.2,2.5);
\draw (1,1) rectangle (5,2);
\draw (3,1.5) node {$h$};
\begin{scope}[decoration={
	markings,
	mark=at position 0.5 with {\arrow[>=latex]{>}}}
	]
\foreach \x in {1.25,2,4,4.75}
	\draw [-,postaction=decorate] (\x,0) to (\x,1);
\end{scope}
	
\draw (1.65,0.5) node {$\cdots$};
\draw (4.425,0.5) node {$\cdots$};

\begin{scope}[decoration={
	markings,
	mark=at position 0.8 with {\arrow[thick]{]}},
	mark=at position 0.7 with {\arrow[>=latex]{>}},
	mark=at position 0.9 with {\arrow[>=latex]{<}}}
	]

\draw [-,postaction={decorate}] (3.5,0) to [controls= +(90:0.9) and +(90:0.9)] (2.5,0);
\end{scope}

\draw (3.5,-0.25) node {\high};
\draw (2.5,-0.25) node {\high};

\draw (1.65,0) node[anchor=north] {$\underbrace{\hspace{1mm}\high \hspace{1mm} \cdots \hspace{1mm} \high\hspace{1mm}}_{i-1}$};
\draw (4.37,0) node[anchor=north] {$\underbrace{\hspace{4mm}u \hspace{4mm}}_{n-(i+1)}$};

\end{tikzpicture}
\end{center}
which vanishes, since looking at the value of the map on factors $i$ and $i+1$ we have $(\epsilon)(-\varphi \otimes \text{Id}_{V})(\high\high)=\epsilon(\overline{\low}\high)=0$.
\end{proof}

\subsection{Constructing a spanning set for $Diag_0(n)$.}
\label{sec:yn_span_0}
We take $Y_0(n)=X_0(n)$. We will describe a list of properties that give an implicit description of the terms of $Y_0(n)$.

\begin{prop} A diagram $U \in Diag_0(n)$ is a term of $Y_0(n)$ if it satisfies the list of properties $\mathcal{P}$:
\begin{enumerate}[($P_1$)]
\item Every \dotmap\hspace{1.5mm}is in the closure of the sky.
\item There are no crossings in $U$.
\item Any strand component in $U$ has at most $1$ vertex.
\item All strand components are attached to the ground and positively oriented at the ground.
\item Every bracket is directed towards (encloses) the endpoint of its strand component which is furthest clockise with resepect to the $\star$ (rightmost along the ground).

\end{enumerate}
\end{prop}
\begin{proof} It is clear that terms of $Y_0(n)$ satisfy $\mathcal{P}$. We need to show any diagram which satisfies $\mathcal{P}$ can be constructed by the boxes $R,L,A,B$. We know each strand component is connected to the ground and positively oriented at the ground $(P_4)$. We will scan from left to right along the ground, and each time we encounter a strand end follow the strand to see which vertex it terminates in. We build a word $w \in L_{N_0}$ by assigning a letter to each strand end depending on which of the following cases we have: 
\begin{enumerate}
\item The strand terminates at a dot : $\bullet$.
\item The strand terminates in a bracket 
	\begin{enumerate}
	\item and the other strand terminating in this bracket ends to the right of the current strand end: \lcap
	\item and the other strand terminating in this bracket ends to the left of the current strand end: \rcap
	\end{enumerate}
\end{enumerate}
The properties $\mathcal{P}$ imply that one of the cases above occurs. Now use the bijection between $L_{N_0}$ and $L_{M_0}$ to get a word $w'\in L_{M_0}$ (in the alphabet $R,L,A,B$) whose image under $\delta_n$ is planar isotopic to the diagram we started with.
\end{proof}
\begin{prop}
$Y_0(n)$ spans $Diag_0(n)$.
\end{prop}
\begin{proof}
We perform the following algorithm to reduce an arbitrary picture to the span of $Y_0(n)$:
\begin{enumerate}
\item[(1)] Pull all dots to the sky via symmetric isotopy.
\item[(2)] Use relation $E_5$ to remove all crossings.
\item[(3)] Use $E_1$ and $E_4$ to reduce the number of brackets on any strand component to at most $1$
\item[(4)] Use $E_3$ to remove any floating circles and $E_2$ to remove any double dots (these are the only possible strand components with no attachment to the ground).
\item[(5)] Use $E_4$ to direct every bracket towards the rightmost endpoint of its strand component.
\end{enumerate}

By construction, step $n$ of the algorithm ensures that every term of the linear combination of diagrams produced has property $P_n$ from Definition $7$. What we need to prove is that after each step $n$, $P_m$ will remain true for all terms produced and all $m<n$. This implies any term produced by the algorithm will satisfy all properties of Definition 7.

\begin{itemize}[]
\item Step 2 preserves $P_1$: Consider for each existing \dotmap\hspace{1.5mm}a path to the $\star$ which crosses no strand component. These paths will still cross no strand component after step 2 is performed, so each prior \dotmap\hspace{1.5mm}remains in the closure of the sky. Since no new \dotmap\hspace{1.5mm}is created, $P_1$ is preserved.
\item Step 3 preserves $P_1,P_2$: The underlying undecorated graph is unaffected, so $P_1$ and $P_2$ are preserved.
\item Step 4 preserves $P_1,P_2,P_3$: Circles and double dots are replaced by a constant, and removing strand components can not affect any of $P_1,P_2,$ or $P_3$ since no new crossings, \dotmap, or vertex can be introduced.
\item Step 5 preserves $P_1,P_2,P_3,P_4$: The underlying undecorated graph is unaffected, so $P_1$,$P_2$, and $P_4$ are preserved. $P_3$ is preserved since reversing bracket direction does not affect the number of vertices.
\end{itemize}
\end{proof}

\subsection{Using $L_{N_0}$ for computation in $Alg_0$}
To aid in computation, we use the language $L_{N_0}$ of Section $2.2$ to describe morphisms in $Alg_0$. Identify \dotmap\hspace{1.5mm}with $T(G_1):V \rightarrow \mathbbm{1}$, and \lcap\rcap\hspace{1.5mm}with $T(G_2): V^{\otimes 2} \rightarrow \mathbbm{1}$. Note that \dotmap\hspace{1.5mm}is the projection $v_1^*$, and that \lcap\rcap\hspace{1.5mm}is the determinant map. We note that $L_{N_0}(n)$ is the basis $T(Y_n)$ given in the previous section.

Values for maps in $L_{N_0}(2)$ on the basis $Z_2$ of $V^{\otimes 2}$ are:

\begin{center}
\begin{tabular}{| c | c | c |}
  \hline
   & \dotmap\dotmap & \lcap\rcap \\
  \hline			
  \textbf{00} & 0 & 0 \\
  \hline
  \textbf{01} & 0 & 1 \\
  \hline
  \textbf{10} & 0 & -1 \\
  \hline
  \textbf{11} & 1 & 0 \\
  \hline

\end{tabular}
\end{center}

Values for maps in $B_0(3)$ on the basis $Z_3$ of $V^{\otimes 3}$ are:

\begin{center}
\begin{tabular}{| c | c | c | c | c |}
  \hline
   & \dotmap\dotmap\dotmap & \dotmap\lcap\rcap & \lcap\dotmap\rcap & \lcap\rcap\dotmap \\
  \hline			
  \textbf{000} & 0 & 0 & 0 & 0 \\
  \hline
  \textbf{001} & 0 & 0 & 0 & 0 \\
  \hline
  \textbf{010} & 0 & 0 & 0 & 0\\
  \hline
  \textbf{100} & 0 & 0 & 0 & 0\\
  \hline
  \textbf{110} & 0 & -1 & -1 & 0\\
  \hline
  \textbf{101} & 0 & 1 & 0 & -1\\
  \hline
  \textbf{011} & 0 & 0 & 1 & 1\\
  \hline
  \textbf{111} & 1 & 0 & 0 & 0\\
  \hline  
\end{tabular}
\end{center}

We can see that \lcap\dotmap\rcap$=$\dotmap\lcap\rcap$+$\lcap\rcap\dotmap, which reflects the crossing relation in $Diag_0$, and that the maps of $L_{N_0}(3)=$ \{\dotmap\dotmap\dotmap,\dotmap\lcap\rcap,\lcap\rcap\dotmap\} form a basis of $B_0(3)$.

\section{$\text{URep}_ {\mathbb{F}_p}(\mathbb{F}_p^+)$}
\label{ch:charp}
Let $V_n=({\mathbb{F}_p}^n,\phi_n)$ where $\phi_n: \mathbb{F}_p^+ \rightarrow GL({\mathbb{F}_p}^n)$ is defined by $\phi_n(1)=J_n$, and $J_n$ is the Jordan block of dimension $n$ with eigenvalue $1$. In this section we set $V=V_2$ and study $Alg_p=\mathcal{P}(Rep_{\mathbb{F}_p}(\mathbb{F}_p^+,V))$. Taking the standard basis of ${\mathbb{F}_p}^2$, $v_0 = (1,0)$ and  $v_1 = (0,1)$, we use the isomorphism $\varphi:V \rightarrow V^{\ast}$ defined by $\varphi(v_0)=v_1^{\ast}$, $\varphi(v_1)=-v_0^{\ast}$. We will construct a diagrammatic planar algebra $Diag_p$, and an isomorphism of planar algebras to $Alg_p$.

\subsection{Defining $Diag_p$ and the map to $Alg_p$}
For all $p$ we include the generators and relations from $Diag_0$ in our presentation of $Diag_p$. For each $p$ we need one new generator and one new relation, both dependent on $p$.

\begin{mydef}
We present $\text{Diag}_p$ over $\mathbb{F}_p$ by the same generators and relations as $Diag_0$ along with the new generator $G_4$, which we require to be symmetric through relations $E_{s_1}$ and $E_{s_2}$, and relation $E_6$. We call $G_4$ the \textbf{jellyfish}.

$$G4: \underbrace{\vinclude{jellyfish}}_{2p-1}$$

\begin{center}
\vinclude{dotjellyfish} \vinclude{sym_p-1}
\end{center}

\begin{center}
\vinclude{jellyfish12} $\overset{\scriptscriptstyle{E_{s_1}}}{=}$ \vinclude{jellyfish} $\overset{\scriptscriptstyle{E_{s_2}}}{=}$ \vinclude{jellyfish1n}
\end{center}

\end{mydef}

We need to choose a map of representations in $Alg_p$ to be the image of $G_4$.

\begin{mydef}\label{jellymap} $V^{\otimes{n}}$ has basis $Z_n=\{v_0,v_1\}^{\otimes n}$ indexed by $B=\{0,1\}^{\times n}$. Define the length of a basis vector, $l_n:Z_n \rightarrow \mathbb{N}$, by $l(z_b) = \sum_i b_i$, and define $W_i=l^{-1}(i)$. Define a linear map $j_p: V ^{\otimes 2p-1} \rightarrow \mathbb{F}_p$ via its values on $Z_n$:
$$
j_p(v) = \left\{
        \begin{array}{ll}
            1 & \quad v \in W_{p-1},W_{2(p-1)} \\
            0 & \quad else
        \end{array}
    \right.
$$
\end{mydef}

In fact $j_p$ is a map of representations; take $q=p,n=2p-1$ in Proposition 18 of the appendix.

\begin{prop} For each prime $p$ there is a map of planar algebras $T_p:\text{Diag}_p \rightarrow Alg_p$ determined by the values $T_p(G_1)=v_1^{\ast},T_p(G_2)=\varphi,T_p(G_3)=\varphi^{-1}$ and $T_p(G_4)=j_p$, where $j_p$ is as in Definition \ref{jellymap}.
\end{prop}
\begin{proof} The computations for relations $E_1$ through $E_5$ performed in $Alg_0$ all hold in any field, so we just need to check that $E_{s_1}$, $E_{s_2}$, and $E_6$ are satisfied.
\begin{itemize}[]
\item $E_{s_1}$ and $E_{s_2}$ are satisfied: $j_p$ is symmetric by definition, as the length of a basis vector is invariant under permutation of the factors.

\item $E_6$ is satisfied: The image of the LHS vanishes on any vector of $Z_i$ when $i \neq p-1$, and takes value $1$ on any vector of $Z_{p-1}$. We show the image of the RHS has the same property. Suppose $v \in Z_i$ and $i \neq p-1$. Then as the number of factors equal to $\high$ and $\low$ in $v$ are not the same, some cap will have input $\low\low$ or $\high\high$, therefore the map will vanish. Now suppose $v \in Z_{p-1}$. The image of the RHS is invariant under transposing any pair of the first $p-1$ factors or last $p-1$ factors by definition. To show it is symmetric, we need to show invariance under transposing factors $p$ and $p+1$. This follows from applying the crossing relation $E_5$, and using the fact that the right partial trace of the symmetrizer on $p-1$ strands is $0$ in characteristic $p$ \cite{JWformula}. Now by symmetry we may assume $v=\high^{p-1} \cdot \low^{p-1}$, so each term of the symmetrizer will take value $1$, and as there are $(p-1)!$ terms we get $\frac{(p-1)!}{(p-1)!}=1$.
\end{itemize}

\end{proof}

\subsection{Constructing $X_p(n)$ and showing independence of $T_p(X_n)$}
The indecomposable representations that appear in $\otimes$ powers of $V$ are enumerated by the sequence $(V_n)_{n \in [1..p]}$. We have $V_2 \otimes V_n \simeq V_{n-1} \oplus V_{n+1}$ when $1 < n < p$, and $V_2 \otimes V_p \simeq V_p \oplus V_p$, so every indecomposable appears as a direct summand of some $V_2^{\otimes n}$ and all indecomposables are self dual. A nice exposition of this and further references are found in \cite{INVpBook}. We get the following fusion graph $\Gamma_p(V)$.
\begin{center}
\includegraphics[scale=0.9]{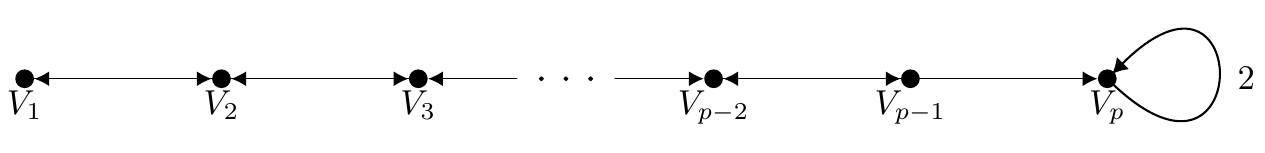}
\end{center}
Note that this is the underlying graph for the automata $M_p$ of Section $2.2$. We consider the corresponding language $L_{M_p}$.


\begin{mydef}
To define a map $\delta_{n,p} :L_{M_p}(n) \rightarrow \text{Diag}_p(n)$ identify concatenation in a path word with gluing of diagrams, and identify each letter of the alphabet with a portion of a picture as below, where $\emptyset$ signifies the end of the word. As before we glue these portions of images end to end to get the image of a word.

\begin{figure}[H]

\hspace*{\fill}
\includegraphics{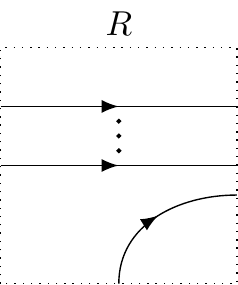} \hfill \includegraphics{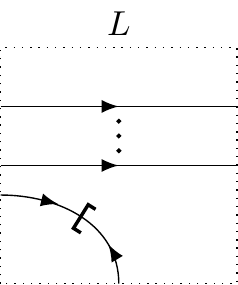} \hfill \includegraphics{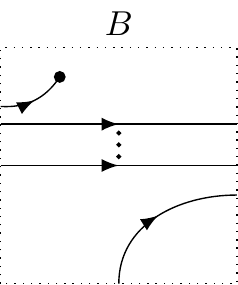} \hfill \includegraphics{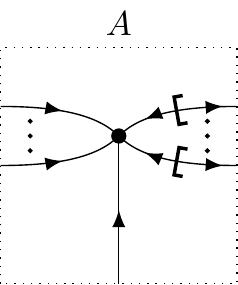} \hfill \includegraphics{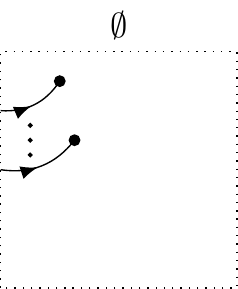}
\hspace*{\fill}

\caption{The image of individual letters under $\delta_{n,p}$}
\end{figure}

\begin{remark}
The number of horiztonal strands in $R$ and $L$ varies, and is equal to the excess number of $R$s to $L$s in the segment prior to the current letter of $w$. Both $B$ and $A$ always have $p-1$ strands incoming and outgoing. The number of incoming strands to $\emptyset$ is the excess of $R$s to $L$s in $w$.
\end{remark}

We will make a distinction between the strands at the left, bottom, and right boundary faces of our diagrams by referring to them as the \textbf{left input} ($I_{\leftarrow}$), \textbf{bottom input} ($I_{\downarrow}$), and \textbf{output} ($O$) so that each diagram can be interpreted as a map from $I_{\leftarrow} \otimes I_{\downarrow}$ to $O$. In particular we have:  \begin{align*}
T_{p}(R) &: V^{\otimes l} \otimes V \rightarrow V^{\otimes l+1} \\
T_{p}(L) &: V^{\otimes l+1} \otimes V \rightarrow V^{\otimes l}\\
T_{p}(B) &: V^{\otimes p-1} \otimes V \rightarrow V^{\otimes p-1} \\
T_{p}(A) &: V^{\otimes p-1} \otimes V \rightarrow V^{\otimes p-1}
\end{align*}

When the context is clear we will use $X$ and $T_p(X)$ interchangeably. The image of $\delta_{n,p}$ forms a graded subcategory in $Diag_p$, where the grade of $\delta_{n,p}(w)$ is equal to the length of $w$ (i.e. the number of bottom inputs). The gluing operation in $Diag_p$ clearly respects this grading, and composition of morphisms $f:V^{\otimes l} \otimes V^{\otimes n} \rightarrow V^{\otimes j}$ and $g:V^{\otimes h} \otimes V^{\otimes m} \rightarrow V^{\otimes k}$ is defined when $j=h$, giving a morphism from $l$ to $k$ in grade $n+m$.  
\end{mydef}

In the figure below we see an example where $w=RRABA$ when $p=3$.
\begin{figure}[H]
\centering
\includegraphics[scale=0.9]{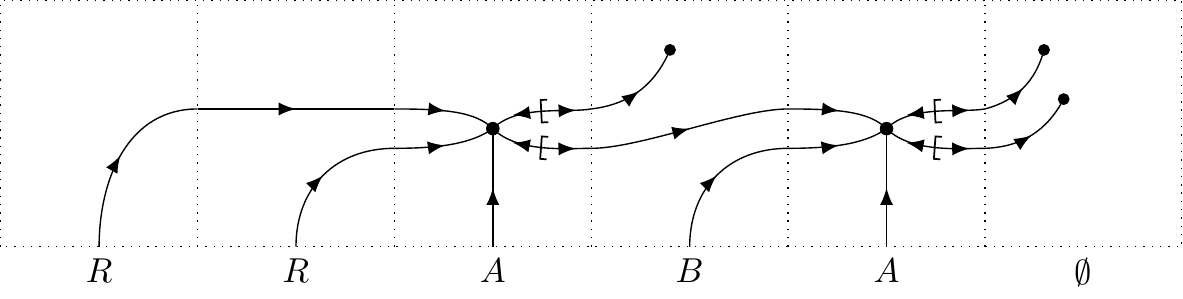}
\caption{The image of $RRABA$ under $\delta_{n,p}$ for $n=6,p=3$}
 \end{figure}

The images of $\delta_{n,p}$ form our sequence $X_p(n)$. We give an order on $X_p(n)$, and a few useful lemmas before showing independence of $T_p(X_n)$.

\begin{mydef} Define an order on $L_{M_p}(n)$, denoted $<_p$, as follows. First order lexicographically with $R>_p L$. Then, among two words with the same $RL$ part, order reverse lexicographically with $B<_p A$. Push this order to $X_p(n)$ via $\delta_{n,p}$, i.e. $x <_p y \iff \delta_{n,p}(x) <_p \delta_{n,p}(y)$.
\end{mydef}

For example, take $p=3,n=5$. We get the following words in $L_{M_3}(5)$ listed in ascending order (using color to separate the $RL$ part and the $AB$ part):

\noindent \color{blue}RLRLR\color{black} $<_p$ \color{blue}RLRR\color{black}B $<_p$ \color{blue}RLRR\color{black}A
$<_p$ \color{blue}RR\color{black}BBB $<_p$ \color{blue}RR\color{black}ABB $<_p$ \color{blue}RR\color{black}BAB
 $<_p$ \color{blue}RR\color{black}AAB $<_p$ \color{blue}RR\color{black}BBA $<_p$ \color{blue}RR\color{black}ABA
 $<_p$ \color{blue}RR\color{black}BAA $<_p$ \color{blue}RR\color{black}AAA

\begin{lemma}\label{computes} The following hold for all $\alpha$ in $\mathbb{F}_p$, where the subscript on a map indicates the bottom input.
\begin{enumerate}
\item $A_\low(\high^{p-1} + \alpha \low^{\hspace{0.5mm}p-1})=\high^{p-1} + (\alpha + 1) \cdot \low^{\hspace{0.5mm}p-1}$
\item $B_\high(\high^{p-1} + \alpha \low^{\hspace{0.5mm}p-1})=\high^{p-1}$
\end{enumerate}
\end{lemma}
\begin{proof}

\begin{enumerate}
\item The jellyfish takes value $1$ on a basis vector when that vector has $p-1$ or $2(p-1)$ factors equal to $\high$, and vanishes otherwise. Further, $\varphi(\low)=\high,\varphi(\high)=-\low$ so that $A_\low(\high^{p-1})=\high^{p-1} + \low^{p-1}$ and $A_\low(\low^{p-1})=\low^{p-1}$, and the result follows by linearity of $A$.
\item $B$ vanishes when a $\low$ is set to the first factor of its left input, so $B_\high(\alpha $\low$^{p-1})=0$.  When a $\high$ is set to the first factor of its left input we have $B_x(\high y)=yx$, so $B_\high(\high^{p-1})=\high^{p-1}$, and the result follows by linearity of $B$.
\end{enumerate}
\end{proof}

\begin{lemma}\label{spaninv}
 Let $Z^-$ be the set of all standard basis vectors of $V^{\otimes p-1}$ of length less than $p-1$ as in Definition \ref{bits} (i.e. all standard basis vectors excluding $\high^{\otimes p-1}$). The span of $Z^-$ is invariant under $A_\low$ and $B_\high$.
\end{lemma}
\begin{proof}
$span(Z^-)$ is invariant under $B_\high$, since $B_\high(\high \otimes v) = v \otimes \high$ and $B_\high(\low \otimes v)=0$, so that $B_\high$ preserves the length of a basis vector or vanishes. Suppose some basis vector $v$ has length $i$. Then since the jellyfish is defined to take value $1$ on basis vectors of length $p-1$ or $2(p-1)$ and vanish otherwise, $A_\low(v)$ is the sum of all basis vectors of length $j$ such that $i + (p-1-j)=p-1$ or $i+(p-1-j)=2(p-1)$ (we add $p-1-j$ to $i$ instead of $j$, since we apply $\varphi$ to each of the factors in the right output), so that must have $j=i$ or $j=i-(p-1)$. We see any basis vector appearing in $A_\low(v)$ has length less than or equal to the length of $v$, so that $span(Z^-)$ is invariant under $A_\low$.
\end{proof}

\begin{prop} The sequence $T_p(X_n)$ is linearly independent.
\end{prop}

\begin{proof} Following the notation of Lemma 2, take $S=(L_{M_p},<_p)$ and $g=T_p \circ \delta_{n,p}$. To define $f$, assign to each $x\in L_{M_p}$ a vector $f(x)\in V^{\otimes n}$ by identifying $R$ and $B$ with \high, $L$ and $A$ with \low, and concatenation with $\otimes$ (e.g. for $p=3$, $f(RLRRABBA)=$\high\low\high\high\low\high\high\low). We prove each point of Lemma \ref{indep}:

\begin{enumerate}

\item $g(x)(f(x)) \neq 0$: If no $A$ or $B$ appears in $x$, the same proof holds as in the characteristic $0$ case. Otherwise, as each $R$ pairs with \high, the first $A$ or $B$ will have left input of $\high^{p-1}$ since $\overset{\leftarrow}{x}$ has depth $p-1$. By (1) and (2) of Lemma \ref{computes}, the left input of $\emptyset$ will be $\high^{p-1}+ \alpha \low^{p-1}$ for some $\alpha \in \mathbb{F}_p$. Applying $\emptyset$ to $\high^{p-1}+ \alpha \low^{p-1}$ we get $1$, so $g(x)(f(x))=1$ for each $x \in S$. 

\item $x<y \implies g(x)(f(y))=0$: If $\overset{\leftarrow}{x} \neq \overset{\leftarrow}{y}$, we proceed as in the proof of Proposition \ref{prop:indep0}. If $\overset{\leftarrow}{x} = \overset{\leftarrow}{y}$ and $x<y$, then we have

\begin{align*} 
f(y) &= s \cdot \low \cdot v \\
f(x) &= s' \cdot \high \cdot v
\end{align*}

We show in $g(x)(f(y))$ that some $\low$ will appear in the left input to $\emptyset$, therefore $g(x)(f(x))=0$ since the value of \dotmap\hspace{1.5mm}on \low\hspace{1.5mm}is $0$. In fact, the only basis vector on which $\emptyset$ does not vanish is $\high^{p-1}$. Since we put a \low\hspace{1.5mm}in the bottom input of $B$ in $g(x)(f(y))$, the right output will be in span$(Z^-)$. Since the factor $v$ is paired with $g(x)$ as defined by $f$, we will then apply some sequence of $A_\low$s and $B_\high$s before applying $\emptyset$. By Lemma \ref{spaninv} the left input of $\emptyset$ will then be in span$(Z^-)$, and $\emptyset$ vanishes on span$(Z^-)$. 
\end{enumerate}
\end{proof}

\subsection{Constructing a spanning set for $Diag_p(n)$}
\label{sec:yn_span_p}
Consider the language $L_{N_p} (\simeq L_{M_p})$ from Section $2.2$. We use $L_{N_p}(n)$ to index a sequence of diagrams $Y_p(n)$. Then we give a list of properties $\mathcal{P}$ that can be checked for any diagram, and show these properties give an implicit description of the terms of $Y_p(n)$, i.e. a diagram is a term of $Y_p(n)$ exactly when all properties of $\mathcal{P}$ are satisfied. Finally we give an algorithm to write an arbitrary diagram of $\beta_p(n)$ in terms of diagrams satisfying $\mathcal{P}$, so that $Y_p(n)$ spans $\beta_{p}(n)$.

\begin{mydef} As in Definition $15$, we define a map $\gamma_{n,p}: L_{N_p}(n) \rightarrow Diag_p(n)$ using the boxes below. Define $Y_p(n)$ to be the images of $\gamma_{n,p}$.
\end{mydef}

\begin{center}
\hspace*{\fill}
\vinclude{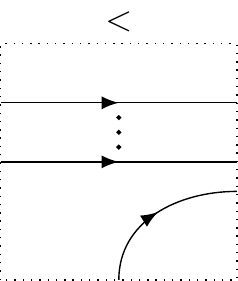} \hfill \vinclude{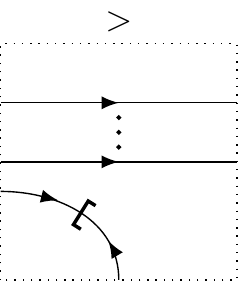} \hfill \vinclude{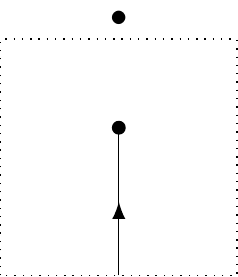} \hfill \vinclude{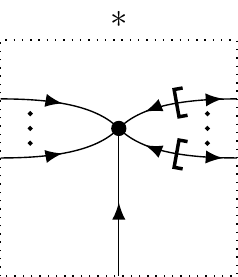}
\hspace*{\fill}
\end{center}

\begin{figure}[H]
\centering
\includegraphics[scale=0.9]{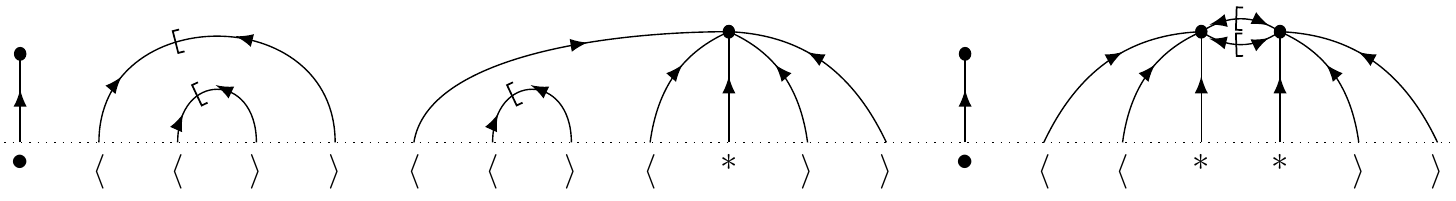}
\caption{An example term from $Y_p(19)$, the image under $\gamma_{n,p}$ of \dotmap\lcap\lcap\rcap\rcap\lcap\lcap\rcap\lcap$*$\rcap\rcap\dotmap\lcap\lcap$**$\rcap\rcap}
\end{figure}

\begin{mydef} Given a diagram $U$ and a point $x$ in the complement of $U$, define the distance from $x$ to the sky to be the minimal number of strand crossings among all paths from $x$ to the sky (these paths should not pass through a vertex). We will say $U \in Diag_p(n)$ satisfies $\mathcal{P}$ if it satisfies:
\begin{enumerate}[($P_1$)]
\item Each dot and jellyfish touches the closure of the sky.
\item There are no crossings in U.
\item The distance from any point of the complement of $U$ to the sky is less than $p$.
\item Any two jellyfish have less than $p$ of their legs connected.
\item No jellyfish has a dot connected to any of its legs.
\item Any strand component or strand terminating at a jellyfish has at most $1$ vertex.
\item All strand components are attached to the ground and positively oriented at the ground.
\item Any bracket encloses the rightmost endpoint of a strand component.

\end{enumerate}
\end{mydef}

\begin{prop}
The terms of $Y_p(n)$ are exactly the diagrams which satisfy $\mathcal{P}$.
\end{prop}
\begin{proof}
It is clear any diagram constructed with the boxes of Definition $16$ satisfy $\mathcal{P}$. In the other direction we know each strand component is connected to the ground and oriented upward from the ground $(P_7)$. We will scan from left to right along the ground, and each time we encounter a strand end follow the strand to see which vertex it terminates in. We build a word $w \in L_{N_p}$ by assigning a letter to each strand end depending on which of the following cases we have: 
\begin{enumerate}
\item The strand terminates at a dot : $\bullet$.
\item The strand terminates at a jellyfish 
	\begin{enumerate}
	\item and is the $p^{\text{th}}$ (middle) leg counting from a leg touching the sky: $*$.
	\item and is left of the middle leg: \lcap
	\item and is right of the middle leg: \rcap
	\end{enumerate}

\item The strand terminates in a bracket 
	\begin{enumerate}
	\item and the other strand terminating in this bracket ends to the right of the current strand end: \lcap
	\item and the other strand terminating in this bracket ends to the left of the current strand end: \rcap
	\end{enumerate}
\end{enumerate}
The properties $\mathcal{P}$ imply that one of the cases above occurs, and that the image of $w$ is planar isotopic to the diagram we started with.
\end{proof}


Now we give an algorithm to write any diagram as a linear combination of terms from the $Y_n$. First we will need to show some consequences of the defining relations in $Diag_p$ which will be used in the reduction algorithm. To simplify notation, whenever there are many copies of the same strand (including orientation markings, brackets, and $\bullet$), we draw a red ring around one copy of the strand labelled by the number of copies contained in that ring. We denote the disoriented Jones-Wenzl on $p-1$ strands with the identity term removed by $\boxtimes$, and the jellyfish by \begin{tikzpicture}\draw[fill=gray](0,0.2) circle (2.5pt);\end{tikzpicture}.

\begin{prop} The relations below hold in $Diag_p$, and are consequences of the defining relations. 

\begin{enumerate}
\item Strand depth reduction: \vinclude{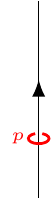} $\overset{\scriptscriptstyle{E_7}}{=}$ \vinclude{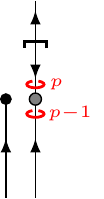} $+$ \vinclude{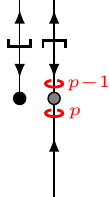} $-$ \vinclude{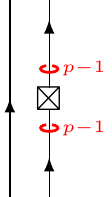} $-$ \vinclude{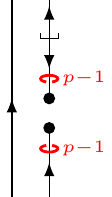}

\item Freeing the dots: \vinclude{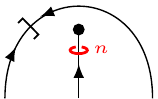} $\overset{\scriptscriptstyle{E_8}}{=} \underset{0\leq i \leq n}{\mathlarger{\sum}}$ \vinclude{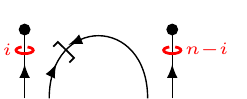}

\item Capping a jellyfish: \vinclude{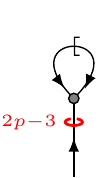} $\overset{\scriptscriptstyle{E_9}}{=} 0$

\item Snipping the legs of $p$-connected jellyfish:\begin{center} \vinclude{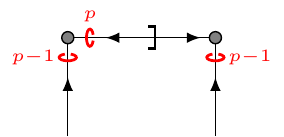} $\overset{\scriptscriptstyle{E_{10}}}{=}$ \vinclude{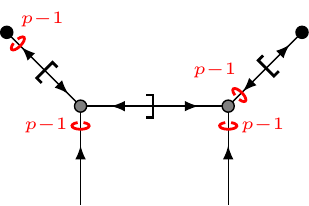}\end{center}
\end{enumerate}
\end{prop}

\begin{proof}
\begin{enumerate}
\item Start with relation $E_{6}$, where we denote the disoriented Jones-Wenzl by $\square$. 
\begin{center}
\vinclude{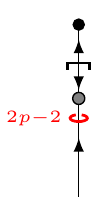} $-$ \vinclude{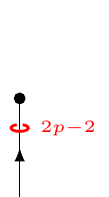} $=$ \vinclude{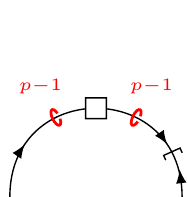}
\end{center}
Now add a cap over both sides of the equation, and use naturality in the first term of the left hand side to pull $\bullet$ and the bracket through the cap. 

\begin{center}
\vinclude{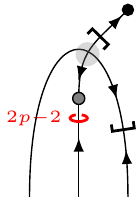} $-$ \vinclude{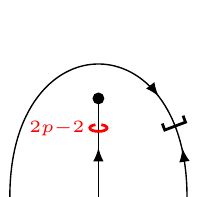} $=$ \vinclude{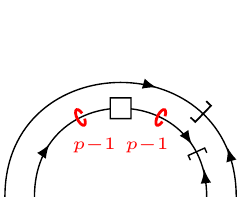}
\end{center}

We then use the crossing relation where indicated and simplify brackets in the result.

\begin{center}
\vinclude{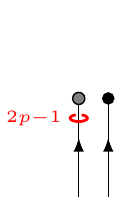} $-$ \vinclude{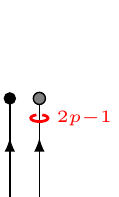} $-$ \vinclude{lotsdotscap} $=$ \vinclude{e10jwcap}
\end{center}

Next use the definition of $\boxtimes$ on the RHS. 

\begin{center}
\vinclude{e10leftpostcap} $-$ \vinclude{e10leftpostcap2} $-$ \vinclude{lotsdotscap} $=$ \vinclude{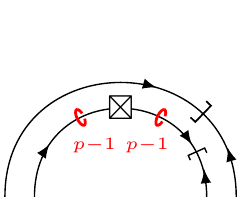} $+$ \vinclude{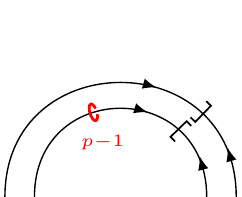}
\end{center}

Finally we apply \vinclude{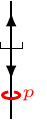} to the rightmost $p$ strands of each term, simplify brackets, and rearrange to give the desired equality.

\item We use induction on $n$. In the base case $n=0$ there is nothing to show. For the induction step use naturality of $\bullet$, and then apply the crossing relation (in the first term of the right hand side we have also used bracket cancellation). 
\begin{center}
\vinclude{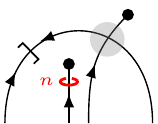} $=$ \vinclude{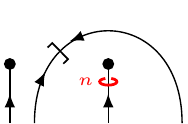} + \vinclude{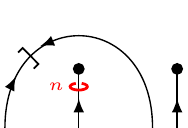}
\end{center}
The result follows from applying the induction hypothesis to each term of the right hand side.

\item Consider the sequence of equalities below. Symmetry of the jellyfish gives equality $1$. Using the crossing relation gives equality $2$, and then equality $3$ follows by bracket reversal and the value of the circle. We see that the capped jellyfish is equal to its negative, implying it is $0$ when $p \neq 2$.
\begin{center}
\vinclude{spidercap} $\underset{\scriptstyle{1}}{=}$ \vinclude{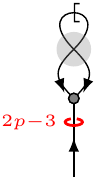} $\underset{\scriptstyle{2}}{=}$ \vinclude{spidercap} $+$\vinclude{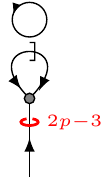} $\underset{\scriptstyle{3}}{=}$ \vinclude{spidercap} $-$ $2$\vinclude{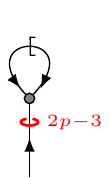} $=$ $-$\vinclude{spidercapR}
\end{center}
If $p=2$ we can use strand depth reduction on the pair of strands flowing from the bracket to the jellyfish. Using the defining relations of $Diag_p$, showing that all terms cancel is straightforward.
\item We apply strand depth reduction to the shared legs of the jellyfish. The term with $\boxtimes$ between the jellyfish vanishes, since every summand of $\boxtimes$ will apply a cap to a jellyfish. In the final term of the RHS we used naturality of $\bullet$ and symmetry of the jellyfish (also reversing all brackets adjacent to the rightmost dot). 
\begin{center}
\vinclude{snipleft} $=$ \vinclude{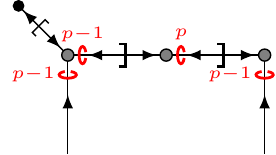} $+$ \vinclude{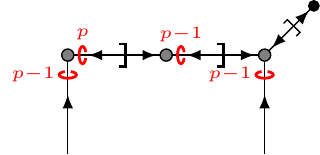} $-$ \vinclude{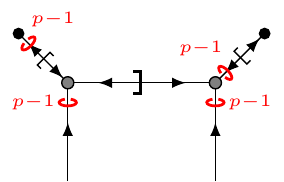}
\end{center}

Now we use $E_6$ on each of the first two terms of the RHS, where the jellyfish has a dot on one of its legs. In the fourth term of the result below we use naturality of the bracket to pull them through the Jones-Wenzl and cancel with the opposite facing brackets on the other side.
\begin{center}
\vinclude{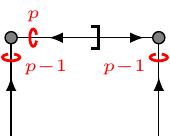} $=$ \vinclude{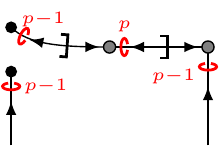} $+$ \vinclude{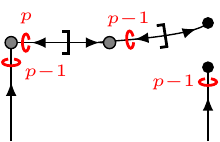} $+$ \vinclude{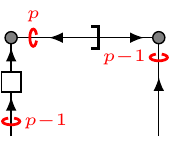} $+$ \vinclude{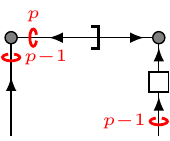} $-$ \vinclude{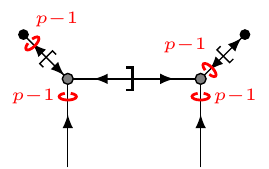}
\end{center}
 In both the third and fourth term of the RHS above we can remove the Jones-Wenzl (symmetrizer) by symmetry of the jellyfish, so that we get $2$ copies of the digram we started with. When $p>2$, the first and second term of the RHS vanish by using $E_6$ again where the jellyfish has dots on $p-1$ of its legs (pick any of the dots to by the dot attached to the jellyfish in the LHS of $E_6$), and noting that every term will contain a double dot or capped jellyfish.

\begin{center}
\vinclude{snipleftnew} $=$ $2$ \vinclude{snipleftnew} $-$ \vinclude{jellycon3small}
\end{center}

Rearranging, we have our result.  If $p=2$, while the computation is a bit different it is straightforward and we will still get the same final result, where further the RHS is just equal to the cap.

\end{enumerate}
\end{proof}

\begin{prop}
Any diagram $U$ of $\beta_p(n)$ is in the span of the $Y_p(n)$ via the following algorithm.
\begin{enumerate}
\item Pull all dots and jellyfish to the sky using naturality.
\item Eliminate any crossings using $E_4$.
\item Apply strand depth reduction ($E_7$) to reduce the distance to the sky for any point whose distance is $p$ or greater. Repeat until $P_3$ is satisfied.
\item Reduce the number of legs shared between any pair of jellyfish to be less than $p$ using the snipping relation $(E_{10})$. Repeat until $P_4$ is satisfied.
\item Remove any jellyfish with a dot connected to any of their legs using $E_6$. Repeat until $P_5$ is satisfied.
\item Use $E_1$ and $E_4$ to reduce the number of brackets on any strand component to at most $1$.
\item Use $E_3$ to remove any floating circles and $E_2$ to remove any double dots (these are the only possible strand components with no attachment to the ground).
\item Use $E_4$ to direct brackets as in $P_8$.

\end{enumerate}
\end{prop}

\begin{proof}

We first point out why all steps, and therefore the algorithm as a whole terminates. Step 3 terminates since it strictly reduces the number of connecting strands between jellyfish without introducing any new jellyfish. Step 4 terminates since it strictly reduces the shortest path to the sky for any given point in a region of the diagram, each region is path connected, and there are finitely many regions. Step 5 terminates since it strictly reduces the number of jellyfish with dots on their legs. The rest of the steps are finite by definition, so the algorithm terminates. Now we proceed as in Section $3$, again noting that the steps are constructed to achieve exactly the properties in the definition of $Y_p(n)$.

\begin{itemize}[]
\item Step 2 preserves $P_1$: No dots or jellyfish are introduced, and the ones that exist will still touch the sky, as only their legs are altered.
\item Step 3 preserves $P_1$ and $P_2$: Always apply this relation on a stack of $p$ strands whose top strand is adjacent to the sky, so that the dots and jellyfish introduced touch the sky. No crossings are introduced. 
\item Step 4 preserves $P_1$ through $P_3$: No crossings are introduced, and the jellyfish of the LHS were assumed to be touching the sky at this point, so the dots and jellyfish of the RHS must touch the sky as well. Lengths of paths to the sky can only be decresed by this relation.

\item Step 5 preserves $P_1$ through $P_4$: No crossings are introduced, and dots introduced are touching the sky since the jellyfish being removed touched the sky. Looking at a point between any of the jellyfish legs, the length of the shortest path to the sky can only decrease, as the jellyfish becomes all dots or some element of $DTL$. No jellyfish are introduced, so $P_4$ is preserved.
\item Steps 6,7,8 preserve $P_1$ through $P_5$: We are removing floating components and flipping or removing brackets from strands, none of which affect any prior property.
\end{itemize}

 
We need to worry about components which are connected jellyfish not attached to the ground, but we claim no such planar graphs exist once the connectivity of jellyfish is below $p$, no jellyfish are connected to dots, all jellyfish are touching the sky, and no jellyfish has a cap. This follows from a graph theoretic argument as in the proof of Theorem 3.8 in \cite{BMPS}. Any such graph has a node with exactly two neighbors, and then at least $p$ of its edges must connect to a neighboring node contradicting the assumption that jellyfish were connected by less than $p$ strands.

\end{proof}

\section{Future directions}
\subsection{Fundamental theorems for rings of vector invariants}
Finding generators and relations for the planar algebra of $Rep_k(G,V)$ is a $\otimes$-version of the first and second fundamental theorems of invariant theory for ${(V^{\oplus n})}^G$, i.e. we are able to compute the subring of multilinear invariants from the planar algebra presentation. This is since maps in $\text{Hom}_G(V^{\otimes n}, \mathbbm{1})$ give multilinear $G$-invariants of $V^{\oplus n}$. Further, a presentation of these spaces gives $\otimes$-versions of the first and second fundamental theorems for ${(V^{\oplus n})}^G$. This is discussed throughout Chapter $5$ of \cite{SymmBook}, and of particular interest is Lemma 5.4.1. The case of $Rep_{\mathbb{C}}(\mathbb{C}^+)$ discussed in this work leads to another proof of the Nowicki conjecture on Weitzenbock derivations as in \cite{NOW}. The key process is to solve the $\otimes$-version of the fundamental theorems, and use the process of polarization and restitution \cite{SymmBook,INVpBook}. In the characteristic $p$ case while there are partial results (such as in \cite{poleresP}), things are more complicated, but the planar algebra results inform the invariant theory.
\subsection{Generalization to $\mathbb{F}_q$ and other generating objects}
The jellyfish is a map of representations of $\mathbb{F}_q^+$ for any finite field $\mathbb{F}_q$ and any number of legs, as in the proposition and proof of the appendix. We would like to generalize from $p$ to $q=p^n$ and give similar results to those given in this thesis. Further we could change the generating object from the standard $2$ dimensional discussed here to other indecomposables.

\section{Appendix}

\subsection{Defining the use of the terminology ground and sky}
\begin{mydef} We call $\delta D_0$ the \textbf{ground}, and the connected component of $\star$ in the complement of a tangle the $\textbf{sky}$.
\end{mydef}

When drawing diagrams we'll assume an isotopy of $D_0$ to a half disc (assume the corners are slightly rounded), which puts the $\star$ of the output disc along the boundary semicircle, and all strands which intersect $\delta D$ on the diameter of the half disc. The sky is shaded in the images below.

\begin{center}
\vinclude{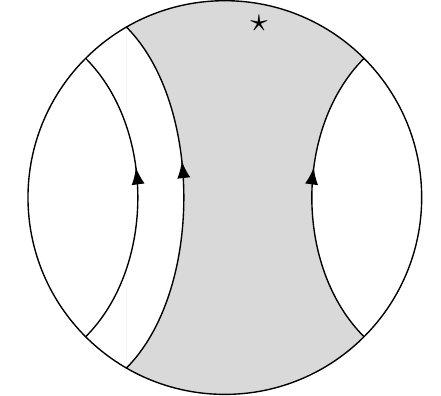} $=$ $\underbrace{\vinclude{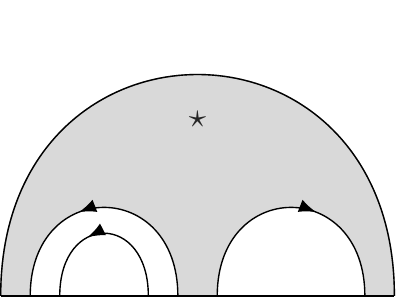}}_{\text{ground}}$
\end{center}

\subsection{$\mathbb{F}_q^+$-invariance of the jellyfish maps}
\begin{mydef}
Let $\mathbb{F}_q$ be a finite field, and $V=\mathbb{F}_q^2$ have basis $(v_0,v_1)$. Make $V$ a representation of the additive group $\mathbb{F}_q^+$ via $x \mapsto \left( \begin{array}{c c} 1 & x \\ 0 & 1 \end{array} \right)$, and take $\mathbb{F}_q$ to be the trivial representation of $\mathbb{F}_q^+$. Using the notations of Definition \ref{bits} define a linear map $j_{q,n}:V^{\otimes n} \rightarrow \mathbb{F}_q$ by its values on $Z_n$:

$$j_{q,n}(z) = \begin{cases} 1 : q-1 | l_n(z), l_n(z) \neq 0,n  \\ 0: else \end{cases} $$
\end{mydef}

\begin{prop}
 For any $n \in \mathbb{N}$, and any $q=p^i$, $j_{q,n}$ is a map of $\mathbb{F}_q^{+}$ representations.
\begin{proof}
We need to show $j_{q,n}(z) = j_{q,n}(x \cdot z)$ $\forall x \in \mathbb{F}_q, \forall z \in Z_n$. We first compute the action of $x \in \mathbb{F}_q$ on some $z \in Z_n$, setting $l=l_n(z)$. We may assume by symmetry of $j_{p,n}$ that $z=\high^l \cdot \low^{n-l}$:
$$
x \cdot z = \underset{i=0}{\sum^{l}}{x^i}\hspace{.5mm}\underset{w\in Z_n(l-i) \hspace{3.5mm}}{\sum{w \cdot \low^{n-l}}}
$$

\noindent Applying $j_{q,n}$ to the above expression gives
$$
j_{q,n}(x\cdot z) = j_{q,n}(z) + \underset{i=1}{\sum^{l - 1}}{x^i}\hspace{0.5mm}\underset{w \in Z_n(l-i) \hspace{3mm}}{\sum{j_{q,n}(w)}}
= j_{q,n}(z) + \sum_{0 < j(q-1) < l}{{l}\choose{j(q-1)}}x^{l-j(q-1)}
$$

\noindent Since $x \in \mathbb{F}_q$, we have $x^{q-1} = 1$ and can simplify the above expression:

$$
j_{q,n}(x\cdot z) = j_{q,n}(z) + x^{l}\sum_{0 < j(q-1) < l}{{l}\choose{j(q-1)}}
$$

We then need to check that $j_{q,n}(z) = j_{q,n}(z) + x^{l}\sum_{0 < j(q-1) < l}{{l}\choose{j(q-1)}}$, or equivalently each $x \in \mathbb{F}_q$ must be a root of $x^{l}\sum_{0 < j(q-1) < l}{{l}\choose{j(q-1)}}$. We see that $0$ is a root, so assume $x \in F_q^{\times}$, and cancelling $x^{l(v)}$ we must show $S = \sum_{0 < j(q-1) < l}{{l}\choose{j(q-1)}} \equiv_p 0$.

The generating function for ${j}\choose{k}$ is $(1 + t)^j$. We would like to exclude the constant term and $t^j$, and then take the sum of coefficients of each $t^{j(q-1)}$. This can be done by fixing a primitive root of unity $g$ of order $q-1$ in $\mathbb{F}_q$, and replacing $t$ by $g^mt$ in $(1+t)^j-(1+t^j)$, then summing over $m$ from $0$ to $q-2$ and evaluating at $t=1$. The result of this is:

$$
\gamma= - \sum_{m=0}^{q-2}[(1+g^m)^{l}-(1 + g^{m \cdot l})]
$$

Now if $q-1$ divides $l$ we know $(1 + g^m)^{l} = 1$ for all but one value of $m$, where $g^m=-1 \implies (1+g^m)^{l}=0$, and $(1+g^{m\cdot l}) = 2$. In this case we get $\gamma = - (-(q-2)-2) = q = 0$.  If $q-1$ does not divide $l$ we have $\sum_{m=0}^{q-2}(1+g^{m\cdot l}) = q-1 = -1$, and $1 + g^m$ will range over $\mathbb{F}_q - \{1\}$. This lets us write

$$
\gamma = -1-\sum_{m=0}^{q-2}(1+g^m)^{l}=-1-\sum_{y \in \mathbb{F}_q-\{1\}}{y^{l}} = \sum_{y \in \mathbb{F}_q}{y^{l}}=\sum_{z \in \mathbb{F}_q}{z} = 0 
$$ 

\end{proof}
\end{prop}


\bibliographystyle{alpha}

\end{document}